\documentclass[12pt]{article}
\usepackage{graphicx,xspace,colortbl}
\usepackage{amsmath,amsthm,amsfonts}
\usepackage{color}
\usepackage{fancybox}
\usepackage{epsfig}
\usepackage{subfig}
\usepackage{pdfsync}
\RequirePackage[colorlinks,citecolor=blue,urlcolor=blue,linkcolor=blue]{hyperref}
\hypersetup{
colorlinks = true,
citecolor=blue,
urlcolor=blue,
linkcolor=blue,
pdfauthor = {Alexey Kuznetsov, Armin Mohammadioroojeh},
pdfkeywords = {},
pdftitle = {Approximating functions on R^+ by exponential sums},
pdfpagemode = UseNone
}
    \def\qed{\hfill$\sqcap\kern-8.0pt\hbox{$\sqcup$}$\\}
    \def\numx#1e#2{{#1}\mathrm{e}{#2}}

    \def\re{\textnormal {Re}}
    \def\im{\textnormal {Im}}
    \def\p{{\mathbb P}}
    \def\r{{\mathbb R}}
    \def\c{{\mathbb C}}
    \def\e{{\mathbb E}}
    \def\d{{\textnormal d}}
    \def\i{{\textnormal i}}

	\newtheorem{proposition}{Proposition}

	\newtheorem{remark}{Remark}

\title{Approximating functions on $\r^+$ by exponential sums}
\author{ 
{Alexey Kuznetsov,  Armin Mohammadioroojeh\footnote{Dept. of Mathematics and Statistics,  York University,
4700 Keele Street, Toronto, ON, M3J 1P3, Canada.  \\
 Email: akuznets@yorku.ca}} }

\date{\today}
 
\begin{document}
%**************************************************************************************************
%**************************************************************************************************
%**************************************************************************************************
\maketitle

\begin{abstract}
We present a new method for approximating real-valued functions on $\r^+$ by linear combinations of exponential functions with complex coefficients. The approach is based on a multi-point Padé approximation of the Laplace transform and employs a highly efficient continued fraction technique to construct the corresponding rational approximant. We demonstrate the accuracy of this method through a variety of examples, including the Gaussian function, probability density functions of the lognormal and Gompertz–Makeham distributions, the hockey stick and unit step functions, as well as a function arising in the approximation of the gamma and Barnes $G$-functions.     
\end{abstract}

{\vskip 0.15cm}
{\vskip 0.25cm}
%**************************************************************************************************
%**************************************************************************************************
%**************************************************************************************************
\section{Introduction}
%**************************************************************************************************
%**************************************************************************************************
%**************************************************************************************************

Let $f$ be a real-valued function defined on an interval $I \subseteq [0,\infty)$. The problem of approximating $f$ by exponential sums of the form
\begin{equation}\label{def:phi}
\phi(x)=\sum_{j=1}^M c_j e^{-\lambda_j x}, \qquad c_j, \lambda_j \in \c,
\end{equation}
is important and has a long history. 
Exponential sum approximations are useful because each term $e^{-\lambda_j x}$ is easy to integrate, convolve, separate across variables, and apply in operator form; this is why such approximations play an important role in fast algorithms and separated representations of kernels and Green’s functions \cite{Beylkin_2005,Beylkin_2010,Hackbusch2019,Jiang2022,Kuznetsov_2022}. The earliest systematic method is due to Gaspard de Prony, dating back to the late eighteenth century. Prony’s method \cite{Fernandez_2018,PlonkaTasche2014} samples the function $f$ on a regular grid of points in the interval $I$ and transforms the problem into one involving a linear difference equation with constant coefficients. Prony’s method is general -- it applies to a wide variety of functions -- and is numerically stable \cite{Katz_2024}. Numerous improvements and variants have since been developed; see \cite{Osborne_1995,Stampfer2020} and the references therein.
In the context of estimating the parameters of a signal represented as a finite sum of exponentials or damped sinusoids, other classical methods include Matrix Pencil \cite{Hua_1990}, MUSIC \cite{Schmidt1986}, and ESPRIT \cite{RoyKailath1989}; see also \cite{DereviankoPlonkaPetz2023,Potts_2013,Stoica_2005}. 

Extensive work has also been devoted to the study of optimal exponential sum approximations in $L_p$ spaces and to the construction of such approximations for specific functions. In this context, a more general class of approximants is often considered, in which the coefficients $c_j$ in \eqref{def:phi} may be polynomials in $x$. Rice \cite{Rice_1962} studied Chebyshev (i.e., $L_{\infty}(I,\d x)$) approximations on a finite interval. Kammler \cite{Kammler_1973, Kammler_1979} characterized optimal $L_p([0,1],\d x)$ approximations and investigated best $L_2((0,\infty),\d x)$ approximations of completely monotone functions. Approximations of power functions $x^{-\beta}$ were studied in \cite{Braess_2009,McLean_2018,Hackbusch_2019}. Further related contributions include the trigonometric approximation of Bessel functions \cite{Cuyt_2020} and the approximation of the Gaussian function by short cosine sums \cite{Derevianko_2023}.

Another widely applicable approach to constructing exponential sum approximations was introduced in 2005 by Beylkin and Monz\'on \cite{Beylkin_2005,Beylkin_2010}. Their method also samples $f$ on a regular grid, but then constructs a Hankel matrix from these values and applies the singular value decomposition. This method has become very popular and has been applied in a wide range of contexts \cite{Feng_2017,Iscoe_2010,McLean_2018}.

In this paper we present a new algorithm for approximating functions $f : [0,\infty) \to \r$ by exponential sums of the form \eqref{def:phi}. Our method is based on a multi-point Padé approximation of the Laplace transform of $f$, and can be viewed as a generalization of the technique introduced in \cite{Kuznetsov_2022}. See 
\cite{DereviankoPlonkaPetz2023} for another approach that uses rational approximations to construct exponential sum approximations. Let $F(z)$ and $R(z)$ denote the Laplace transforms of $f$ and $\phi$, respectively:
\begin{equation}\label{def:F_Laplace_transform}
F(z):=\int_0^{\infty} f(x) e^{-zx} \d x, \qquad
R(z):=\int_0^{\infty} \phi(x) e^{-z x} \d x.
\end{equation}
We assume that $f \in L_1((0,\infty), \d x)$, so that $F$ is well defined in the right half-plane $\re(z) \ge 0$.
If $\phi$ has the form \eqref{def:phi}, then $R(z)$ is a rational function with simple poles at $-\lambda_j$, and the coefficients $c_j$ are obtained from its partial fraction decomposition:
\begin{equation}\label{eqn:Phi_partial_fraction}
R(z)=\sum_{j=1}^M \frac{c_j}{z+\lambda_j}.
\end{equation}

The key observation underlying our approach is that if the exponential sum $\phi$ is close  to $f$ in an appropriate norm, then the rational function $R$ is close  to the Laplace transform $F$, and vice versa. For example, if $\Vert f - \phi \Vert_p$ is small for some $p>1$, then by Hölder’s inequality we have
\begin{align*}
|F(z)-R(z)| &= \Big \vert \int_0^{\infty} (f(x)-\phi(x)) e^{-z x}   \d x \Big \vert \\
&\le \Vert f-\phi \Vert_p \times \Bigg(\int_0^{\infty} e^{-q \re(z) x}  \d x \Bigg)^{1/q}
= \frac{\Vert f-\phi \Vert_p}{(q \re(z))^{1/q}},
\end{align*}
for all $z$ in the right half-plane $\re(z) > 0$, where $q$ is defined by $1/p+1/q=1$. When $p=1$, the same bound holds with the denominator replaced by $1$. Thus, a good exponential sum approximation to $f$ in the $L_p$ norm produces a good rational approximation to $F$ in the $L_{\infty}$ norm. Conversely, writing $f$ and $\phi$ via the inverse Laplace transform yields, for any $c>0$,
$$
|f(x) - \phi(x)| \le \frac{e^{c x}}{2 \pi} \int_{\r} |F(c+\i y) - R(c+\i y)| \, \d y,
$$
so if $F$ is close to $R$ in the $L_1$ norm on some vertical line $c+ \i \r$, then $e^{-c x}(f(x)-\phi(x))$ is small in the $L_{\infty}$ norm. This shows the close connection between rational approximation of Laplace transforms and exponential sum approximation of the original function.

Having established this connection, let us examine what information about the original function $f$ can be readily transferred to its Laplace transform. Suppose that $f$ satisfies
\begin{equation}\label{eqn:f_at_zero}
f(x)=\sum_{j=0}^{n_{\infty}-1} \xi_j\frac{x^j}{j!}+O(x^{n_{\infty}}), \qquad x\to 0^+.
\end{equation}
A sufficient condition for this expansion is the boundedness of the derivative $f^{(n_{\infty})}(x)$ on some interval $(0,c)$; in that case we necessarily have $\xi_j=f^{(j)}(0+)$. Applying Watson's Lemma \cite{Wong_1972}, we obtain
\begin{equation}\label{eqn:F_at_infinity}
F(z) = \sum_{j=0}^{n_{\infty}-1} \xi_j z^{-j-1}
+ O(z^{-n_{\infty}-1}), \qquad z\to \infty,
\end{equation}
uniformly in any sector 
\[
S_{\theta}:=\{ z\in \c \, : \, |\arg(z)|\le \theta\}, \qquad \theta \in (0,\pi/2).
\]
We now require that the approximating rational function $R$ satisfies the same asymptotic condition at $z=\infty$. By Watson's Lemma, this is equivalent to $f(x)-\phi(x)=O(x^{n_{\infty}})$ as $x\to 0^+$. Imposing this condition ensures that $R(z)$ is close to $F(z)$ for large values of $z$.  

To guarantee closeness for moderate and small values of $z$, we impose additional constraints by requiring that $R$ interpolates $F$ at selected points $z_j$ lying in the right half-plane $\re(z) \ge 0$. The general outline of our method is therefore as follows: we aim to construct a rational function $R$ satisfying the following conditions.  

\vspace{0.2cm}

\noindent
{\bf Interpolation conditions at finite points:}
\begin{equation}\label{eqn:Phi_at_z_j}
R(z_j)=F(z_j), \qquad j=1,2,\dots,p, 
\end{equation}
where the points $\{z_j\}_{1\le j \le p}$ lie in the half-plane $\re(z) \ge 0$ and satisfy the symmetry condition $z_{p+1-j}=\overline z_j$ for all $j=1,2,\dots,p$.

\vspace{0.2cm}

\noindent
{\bf Expansion at infinity:}
as $z\to \infty$,
\begin{equation}\label{eqn:Phi_at_infty}
R(z)=\sum_{j=0}^{n_{\infty}-1} \xi_{j}\, z^{-j-1}+O(z^{-n_{\infty}-1}),
\end{equation}
where the coefficients $\xi_j$ are taken from \eqref{eqn:f_at_zero}.  

\vspace{0.2cm}

The above setup defines a multi-point Pad\'e approximation problem with $p+1$ points, namely $\{z_1,z_2,\dots,z_p,\infty\}$. At each finite interpolation point we match only the value of the rational
 function $R$, while at infinity we match the first $n_{\infty}$ coefficients of the power series expansion of $R$ in $z^{-1}$. This algorithm therefore requires knowledge of the first several coefficients of the Taylor expansion of $f$ at zero. These are typically easy to obtain, since in most applications $f$ is given by an explicit formula and its derivatives can be computed directly. The algorithm also requires the values of the Laplace transform of $f$ at the interpolation points $z_j$. When a closed-form expression for $F$ is unavailable, these values must be computed numerically (see Section~\ref{section_Numerical_results} for details).  

A rational function $R(z)$ of the form \eqref{eqn:Phi_partial_fraction} is determined by $2M$ complex parameters $\{c_j, \lambda_j\}_{1\le j \le M}$. The conditions \eqref{eqn:Phi_at_z_j} and \eqref{eqn:Phi_at_infty} impose $p+n_{\infty}$ equations on $R$: $p$ equations from interpolation at $\{z_j\}_{1\le j \le p}$ and $n_{\infty}$ equations ensuring that the coefficients in the expansion of $R$ in powers of $z^{-1}$ coincide with the prescribed $\xi_j$. Thus, we obtain $p+n_{\infty}$ constraints for $2M$ unknowns. It is therefore natural to expect that when $p+n_{\infty}=2M$, there exists at most one rational function $R$ satisfying these conditions. The following proposition confirms this expectation.  

\begin{proposition}\label{prop_unique}
If $p+n_{\infty}=2M$, then there exists at most one rational function $R$ of the form \eqref{eqn:Phi_partial_fraction} satisfying the conditions \eqref{eqn:Phi_at_z_j} and \eqref{eqn:Phi_at_infty}. Moreover, if such a rational function $R$ exists, it is necessarily real.
\end{proposition}

\begin{proof}
Suppose there exist two rational functions $R_1$ and $R_2$ of the form \eqref{eqn:Phi_partial_fraction} that satisfy \eqref{eqn:Phi_at_z_j} and \eqref{eqn:Phi_at_infty}. Each $R_i$ can be written as $R_i(z)=P_i(z)/Q_i(z)$, where $\deg(P_i)<\deg(Q_i)=M$. The fact that both $R_1$ and $R_2$ satisfy the interpolation and asymptotic conditions implies that their difference
\[
G(z):=R_1(z)-R_2(z) = \frac{P_1(z)Q_2(z)-P_2(z)Q_1(z)}{Q_1(z)Q_2(z)}
\]
has $p$ zeros at the points $\{z_j\}_{1 \le j \le p}$ and satisfies $G(z)=O(z^{-n_{\infty}-1})$ as $z\to\infty$. Consequently, the numerator $P_1Q_2-P_2Q_1$ is a polynomial with $p$ distinct zeros, yet
\[
\deg(P_1Q_2-P_2Q_1) \le \deg(Q_1Q_2)-n_{\infty}-1 = 2M-n_{\infty}-1=p-1.
\]
Thus the polynomial must vanish identically, and we conclude that $R_1(z)\equiv R_2(z)$.  

Next, suppose that $R$ is a rational function of the form \eqref{eqn:Phi_partial_fraction} satisfying \eqref{eqn:Phi_at_z_j} and \eqref{eqn:Phi_at_infty}. Define
\[
R_3(z) := \overline{R(\overline{z})},
\]
that is, $R_3$ is obtained from $R$ by conjugating all coefficients. For every $j=1,2,\dots,p$ we then have
\[
R_3(z_j) = \overline{R(\overline{z_j})}
= \overline{F(z_{p+1-j})} = F(z_j) = R(z_j).
\]
Since $\xi_j\in \r$ and the asymptotic condition \eqref{eqn:Phi_at_infty} holds, it follows that $R(z)-R_3(z)=O(z^{-n_{\infty}-1})$ as $z\to\infty$. By the same uniqueness argument as above, we conclude that $R(z)\equiv R_3(z)$ for all $z\in \c$. In particular, $R(z)\in \r$ whenever $z\in \r$, which proves that $R$ is a real rational function.
\end{proof}

The multi-point Pad\'e approximation problem described in \eqref{eqn:Phi_at_z_j} and \eqref{eqn:Phi_at_infty} requires as input two positive integer parameters, $M$ and $n_{\infty}$, $p=2M-n_{\infty}$ complex numbers $\{z_j\}_{1\le j \le p}$, together with the
values of $\{F(z_j)\}_{1\le j\le p}$ and coefficients $\{\xi_j\}_{0\le j < n_{\infty}}$, appearing in \eqref{eqn:Phi_at_z_j} and \eqref{eqn:Phi_at_infty}. Even with the symmetry condition $z_{p+1-j}=\bar z_j$, this still leaves too many degrees of freedom. To simplify the algorithm, we propose the following special choice of interpolation points $z_j$.  We assume $p \ge 2$, take two numbers $A\ge 0$ and $B>0$, and consider the piecewise linear curve $\gamma_{A,B}$ consisting of two intervals:
\[
\gamma_{A,B}:=[A-B \i , 0] \cup [0, A+B \i].
\]
We then define $z_1=A-B\i$, $z_p=A+B\i$, and place the remaining points $\{z_j\}_{2 \le j \le p-1}$ equidistantly along the curve $\gamma_{A,B}$. See Figure~\ref{fig1} for an illustration.  

The motivation for this choice comes from Watson's Lemma \cite{Wong_1972}, which implies that $F$ admits the asymptotic expansion \eqref{eqn:F_at_infinity} for large $z$ in any sector
\[
S_{\theta}:=\{ z\in \c \, : \, |\arg(z)|\le \theta\}, \qquad \theta \in (0,\pi/2).
\]
The farther $z$ lies from the origin within such a sector, the more accurately $R(z)$ approximates $F(z)$ (since $F(z)-R(z)=O(z^{-n_{\infty}-1})$ as $z\to \infty$); see Figure~\ref{fig_1c}. We therefore enforce $p$ interpolation conditions on the boundary $\gamma_{A,B}$ of this sector (where $B/A=\tan(\theta)$). This allows us to control the error $F(z)-R(z)$ not only for large $z$, but also for moderate and small values of $z \in S_{\theta}$. By the maximum modulus principle, if $F$ and $R$ are close for large $z$ and along the boundary of $S_{\theta}$, then they will also be close throughout the interior of this sector. Consequently, it is unnecessary to place interpolation points $z_j$ in the interior of $S_{\theta}$.  

\begin{figure}
\centering
\subfloat[][$p=8$, $A=2$, $B=6$]{\label{fig_1_even}\includegraphics[height =5.5cm]{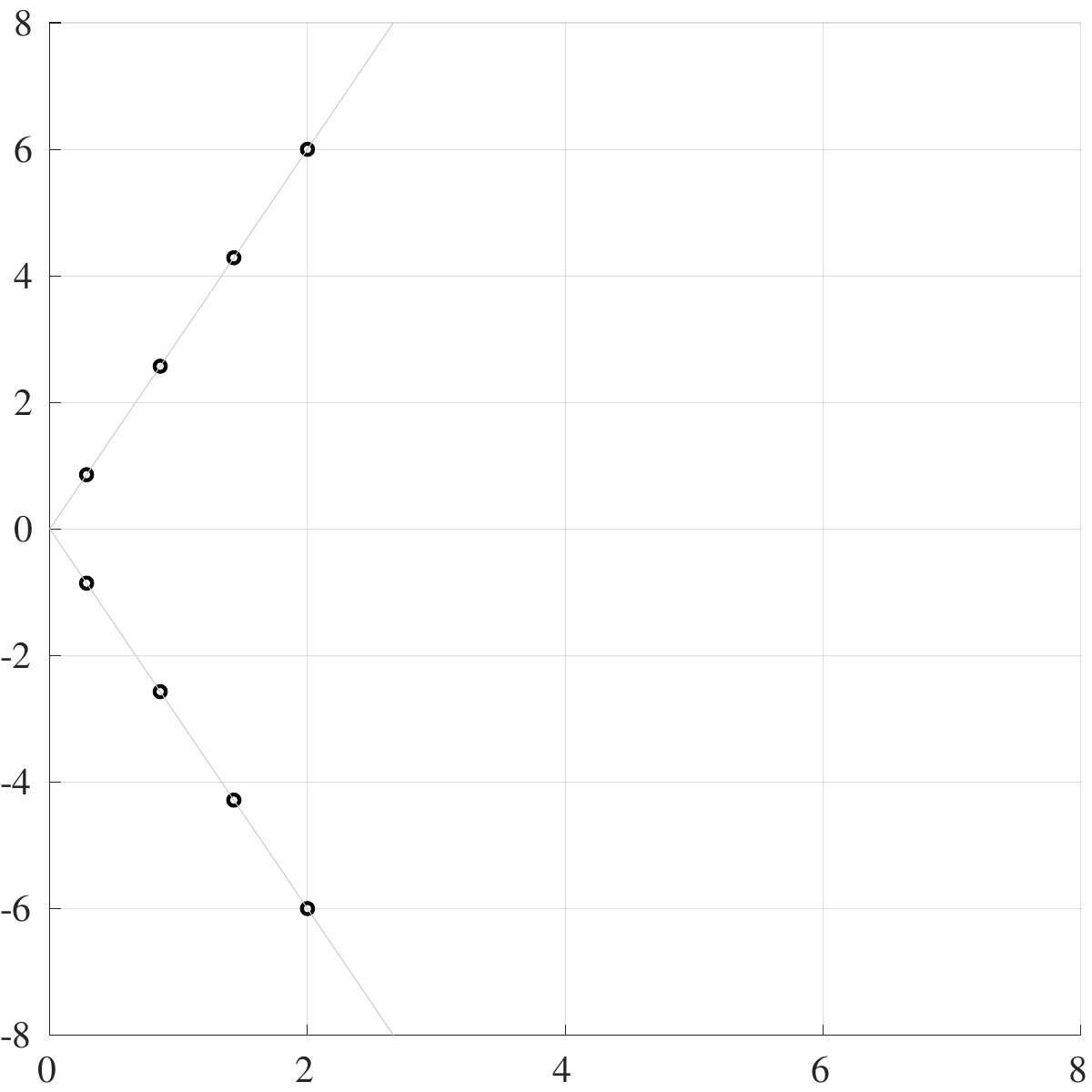}} 
\hspace{0.15cm}
\subfloat[][$p=7$, $A=1$, $B=4$]{\label{fig_1_odd}\includegraphics[height =5.5cm]{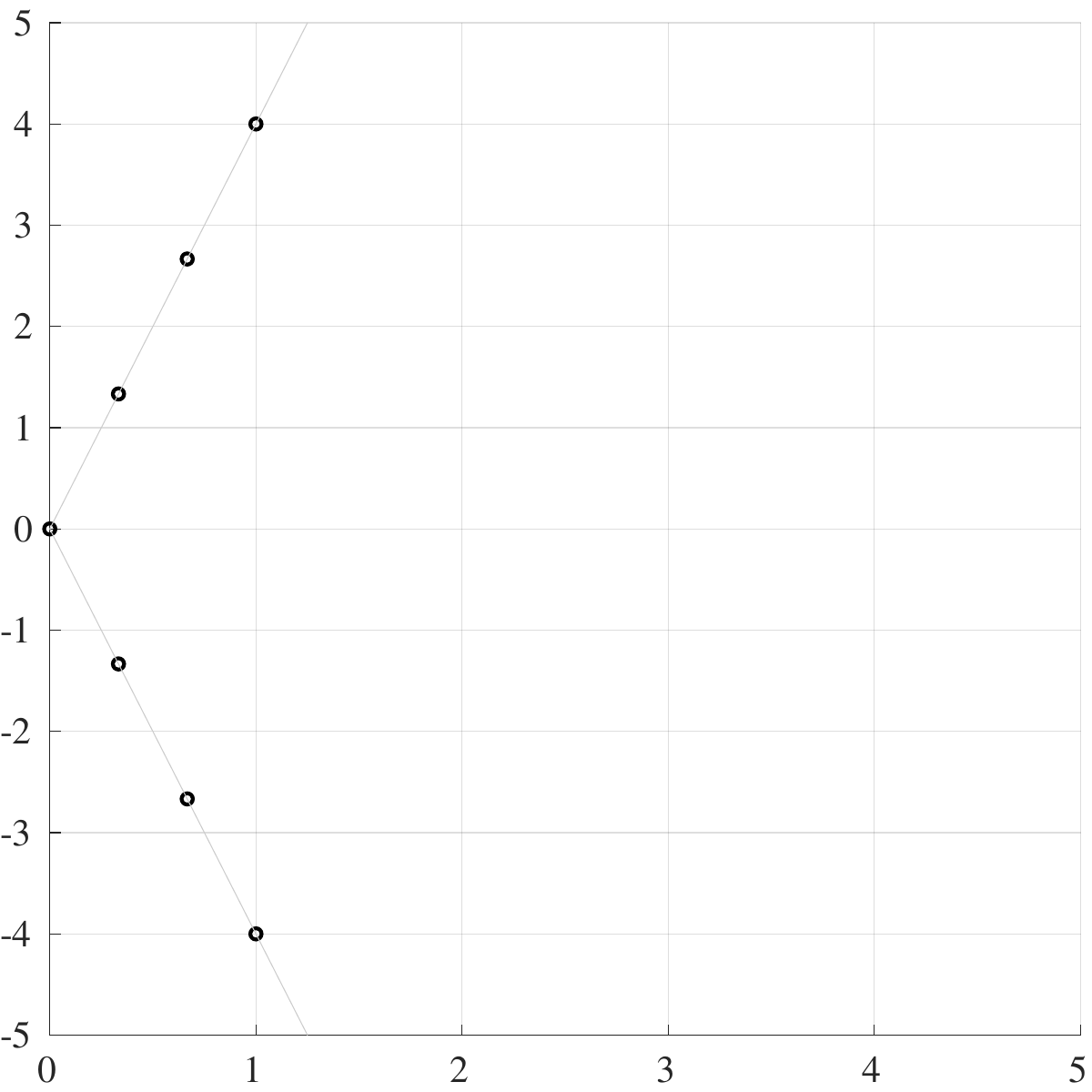}}  
\hspace{0.15cm}
\subfloat[][a sector $S_{\theta} \subset \c$]{\label{fig_1c}\includegraphics[height =5.5cm]{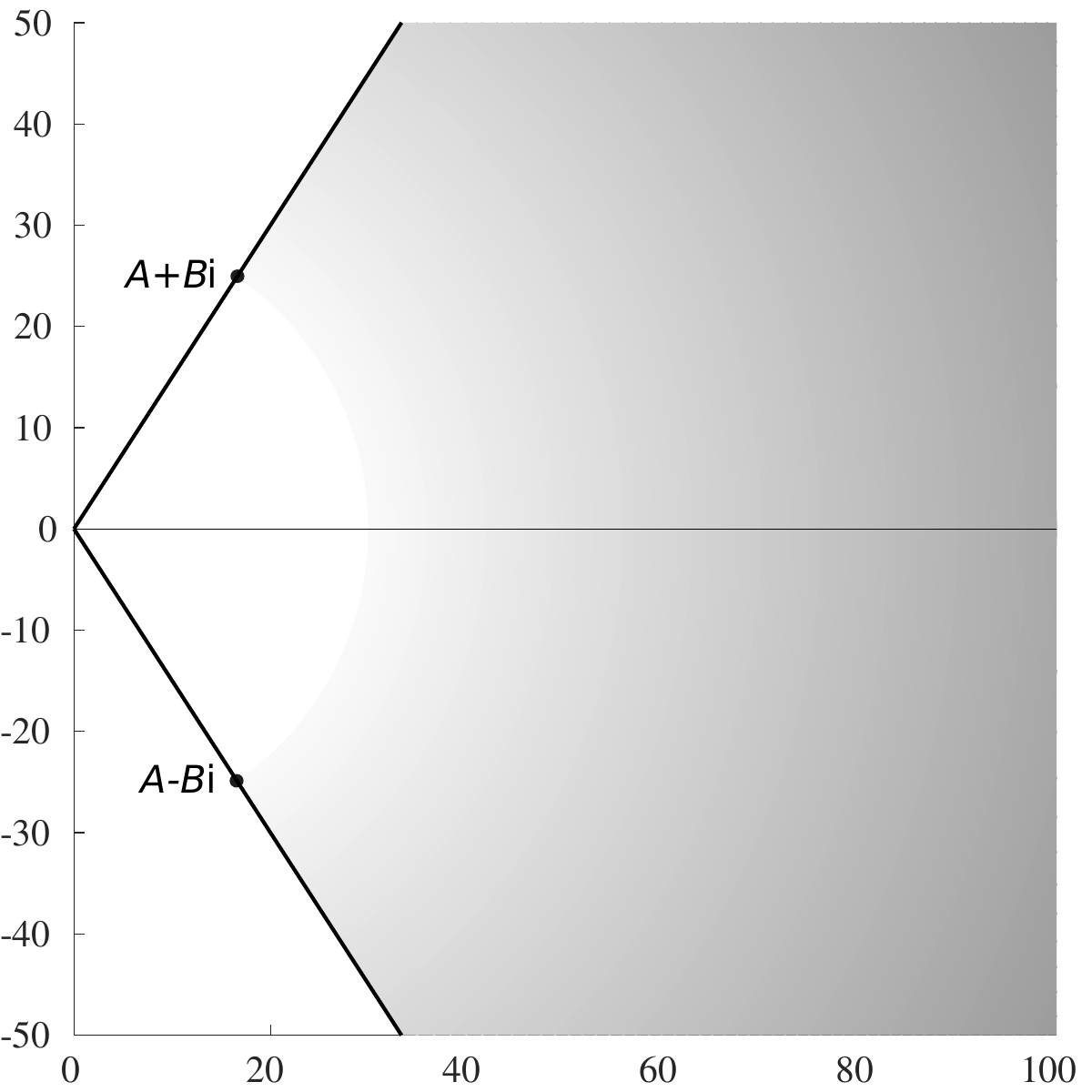}}    
\caption{Configuration of points $\{z_j\}_{1 \le j \le p}$ for even/odd values of $p$. The points are indexed so that $\im(z_j)<\im(z_{j+1})$, and they satisfy the symmetry relation $z_{p+1-j}=\bar z_j$.} 
\label{fig1}
\end{figure}

We can now summarize our algorithm: \label{page_algorithm}
\begin{itemize}
\item[(i)] Choose $M\ge 2$ (the number of terms in the exponential sum approximation $\phi$) and an integer $n_{\infty} \in [1,2M-2]$;
\item[(ii)] Compute the coefficients $\{\xi_j\}_{0\le j < n_{\infty}}$ in the expansion \eqref{eqn:f_at_zero};
\item[(iii)] Choose $A\ge 0$ and $B>0$ and determine $p=2M-n_{\infty}$  interpolation points $z_j$, with $z_1=A-B\i$, $z_{p}=A+B\i$, and the remaining $p-2$ points chosen as shown in Figure~\ref{fig1};
\item[(iv)] Solve the multi-point Pad\'e approximation problem: find a rational function $R$ satisfying the conditions \eqref{eqn:Phi_at_z_j} and \eqref{eqn:Phi_at_infty}; 
\item[(v)] If step (iv) is successful, compute the poles $\{-\lambda_j\}$ of $R$. If all poles are simple, write $R$ in partial fraction form \eqref{eqn:Phi_partial_fraction} and recover the coefficients $c_j$ of the exponential sum $\phi$ given by \eqref{def:phi};
\item[(vi)] Verify whether the exponential sum $\phi$ provides a sufficiently accurate approximation to $f$, using appropriate criteria. If $\phi$ fails this test, go back to step (iii) and choose different values of $A$ and $B$. 
\end{itemize}

All of the above steps are straightforward, with the exception of step (iv). Several methods exist for solving multi-point Pad\'e approximation problems, such as Kronecker's algorithm, Thiele's method, the Thacher-Tukey algorithm, and the generalized Q.D. and $\epsilon$ algorithms. All of these algorithms are described in \cite[Section 1.1]{BakerGravesMorris1981b}. We adopt the continued fraction  approach due to its speed and simplicity: the algorithm is fully recursive and avoids solving large systems of linear equations. The only drawback is that this approach requires the additional assumption  $\xi_0 \neq 0$. This restriction is not severe, however, and in Section~\ref{subsection_Lognormal} we describe a simple workaround. We note here that continued fractions have been applied in the context of Prony's method and signal processing in \cite{Sauer2021}.  

The remainder of the paper is organized as follows. In the next section we present the continued fraction algorithm for solving the multi-point Pad\'e approximation problem in step (iv). Section~\ref{section_Numerical_results} discusses the implementation of our algorithm and reports results from numerous numerical experiments that illustrate both the efficiency of the algorithm and the accuracy of the resulting approximations.  In Section \ref{Section_conclusion} we present some concluding remarks and outline directions for future work.

\section{The continued fraction algorithm}\label{section_Continued_Fractions}

We say that a rational function $R(z)=P(z)/Q(z)$, where $P$ and $Q$ are polynomials, has degree $[n_1/n_2]$ if $\deg(P)=n_1$ and
$\deg(Q)=n_2$. 
In this section we present a continued fraction algorithm for constructing a rational function $R(z)$ of degree $[M-1/M]$ that satisfies $p$ interpolation conditions
\begin{equation}\label{R_conditions_at_zj}
R(z_j)=a_j,  \qquad j=1,2,\dots,p,
\end{equation}
together with an asymptotic expansion at infinity of the form
\begin{equation}\label{eqn:R_at_infty2}
R(z)=\sum_{j=0}^{n_{\infty}-1} \xi_{j} z^{-j-1}+O(z^{-n_{\infty}-1}), \qquad z\to \infty.
\end{equation}
The parameters $p, n_{\infty} \in \mathbb{N}$ (with $p+n_{\infty}=2M$), the interpolation data $\{z_j, a_j\}_{1\le j \le p}$, and the coefficients $\{\xi_j\}_{0\le j \le n_{\infty}-1}$ are assumed to be given, and they satisfy the symmetry conditions $z_{p+1-j}=\bar z_j$, $a_{p+1-j}=\bar a_j$.  

Before describing the general algorithm, we illustrate the idea with a simple example. Let us find a rational function $R$ of degree $[1/2]$ that satisfies
\begin{equation}\label{R_conditions_at_zi}
R(1+\i)=1-2\i, \qquad R(1-\i)=1+2\i,
\end{equation}
and 
\begin{equation}\label{R_conditions_at_infty}
R(z)=z^{-1} + 2 z^{-2} + O(z^{-3}), \qquad z\to \infty.  
\end{equation}
As the first step, we define a function $R_1(z)$ by  
\begin{equation}\label{R_in_terms_of_R1}
R(z)=\frac{R(1+\i)}{1+(z-(1+\i)) R_1(z)}= \frac{1-2\i}{1+(z-1-\i) R_1(z)}. 
\end{equation}
Solving for $R_1(z)$ gives
\begin{equation}\label{R1_in_terms_of_R}
R_1(z)=\frac{1}{z-1-\i}\,\bigg[ \frac{1-2\i}{R(z)} - 1 \bigg].
\end{equation}
From \eqref{R_conditions_at_zi} and \eqref{R1_in_terms_of_R} we obtain
\begin{equation}\label{R1_conditions_at_zi}
R_1(1-\i)=\tfrac{2}{5}-\tfrac{4}{5}\i,
\end{equation}
while from \eqref{R_conditions_at_infty} and \eqref{R1_in_terms_of_R}, expanding both $1/(z-1-\i)$ and $1/R(z)$ in powers of $z^{-1}$, we find
\begin{equation}\label{R1_conditions_at_infty}
R_1(z)=(1-2\i)+3\i z^{-1} + O(z^{-2}), \qquad z\to \infty. 
\end{equation}

We now repeat this process. Using \eqref{R1_conditions_at_zi}, we define $R_2(z)$ by
\begin{equation}\label{R1_in_terms_of_R2}
R_1(z)=\frac{R_1(1-\i)}{1+(z-(1-\i)) R_2(z)}=\frac{2/5-4\i/5}{1+(z-1+\i) R_2(z)}.
\end{equation}
After solving for $R_2(z)$ and using \eqref{R1_conditions_at_infty}, we expand $R_2(z)$ in powers of $z^{-1}$ to obtain
\begin{equation}\label{R2_conditions_at_infty}
R_2(z)=-\tfrac{3}{5} z^{-1}-\big(\tfrac{3}{25}-\tfrac{9}{25}\i\big) z^{-2} + O(z^{-3}), \qquad z\to \infty. 
\end{equation} 
We claim that
\begin{equation}\label{R2_cf_at_infty}
R_2(z)=-\frac{3/5}{\,z-1/5+3\i/5}  
\end{equation}
satisfies condition \eqref{R2_conditions_at_infty}. Indeed, as $z\to \infty$,
\begin{align*}
R_2(z)&=-\frac{3}{5z}\times\frac{1}{1-(1/5-3\i/5)z^{-1}}
=-\frac{3}{5z}\sum_{k\ge 0}(1/5-3\i/5)^k z^{-k}\\
&=-\tfrac{3}{5} z^{-1}-\big(\tfrac{3}{25}-\tfrac{9}{25}\i\big) z^{-2} + O(z^{-3}),  
\end{align*}
which is exactly \eqref{R2_conditions_at_infty}.  

Finally, we perform back-substitution to reconstruct $R$. Substituting $R_2$ given by \eqref{R2_cf_at_infty} into \eqref{R1_in_terms_of_R2} yields $R_1$, which is then substituted into \eqref{R_in_terms_of_R1}. The result is the continued fraction representation
\begin{align*}
R(z)=\cfrac{1-2\i}{1+\cfrac{(2/5-4\i/5)(z-1-\i)}{1-\cfrac{(3/5)(z-1+\i)}{z-1/5+3\i/5}}}.
\end{align*}
This can be written in the equivalent (standard) form 
\[
R(z)=\frac{z+1}{z^2-z+1}.
\]
A direct verification  confirms that this rational function satisfies both \eqref{R_conditions_at_zi} and \eqref{R_conditions_at_infty}. By Proposition~\ref{prop_unique}, we have therefore found the unique solution to the multi-point Pad\'e approximation problem \eqref{R_conditions_at_zi} and \eqref{R_conditions_at_infty}.

We are now ready to present the general algorithm for finding the interpolating rational function. As a first step, we define a rational function $R_1$ by 
\begin{equation}\label{CF_transformation_finite}
R(z)=\frac{a_1}{1+(z-z_1) R_1(z)}, 
\end{equation}
which is equivalent to 
\begin{equation*}
R_1(z)=\frac{1}{z-z_1}\,\bigg[ \frac{a_1}{R(z)} -1 \bigg].
\end{equation*}
By construction, the function $R(z)$ defined through \eqref{CF_transformation_finite} satisfies $R(z_1)=a_1$, as long as $R_1$ does not have a pole at $z_1$. The remaining $p-1$ interpolation conditions for $R$ (at the points $\{z_j\}_{2\le j \le p}$) are satisfied if and only if
\begin{equation}\label{R1_conditions_at_zi2}
R_1(z_{j})=a_j^{(1)}:=\frac{1}{z_{j}-z_1} 
\bigg[ \frac{a_1}{a_{j}}-1 \bigg], \qquad j=2,3,\dots,p. 
\end{equation}
Thus, the linear fractional transformation \eqref{CF_transformation_finite} reduces the problem from interpolation at $p$ finite points to interpolation at $p-1$ points. 

Next, we investigate the effect of the transformation
\eqref{CF_transformation_finite} on the expansion of $R$ as $z \to \infty$. For this purpose it is convenient to assume that $R$ has a series expansion in powers of $z^{-1}$ (valid for sufficiently large $z$):  
\begin{equation*}
R(z)=\sum_{j\ge 0} b_j z^{-j}. 
\end{equation*}
We seek the coefficients $b^{(1)}_j$ of the corresponding expansion of $R_1(z)$: 
\begin{equation*}
R_1(z)=\sum_{j\ge 0} b^{(1)}_j z^{-j}.
\end{equation*}
Rewriting \eqref{CF_transformation_finite} gives
\begin{equation}\label{R_R1}
R(z)-z_1  R(z) R_1(z) +z R(z) R_1(z) = a_1.
\end{equation}
We now obtain equations for $b^{(1)}_j$ by comparing coefficients of powers of $z$ on both sides.  
Comparing coefficients of $z$ in \eqref{R_R1} yields
\begin{equation}\label{eqn_b_b1_1}
b_0 b^{(1)}_0=0.
\end{equation} 
Comparing constant terms (coefficients of $z^{0}$) gives
\begin{equation}\label{eqn_b_b1_0}
b_0- z_1 b_0 b_0^{(1)}+b_0 b_1^{(1)}+b_1 b_0^{(1)}=a_1,
\end{equation}
and for coefficients of $z^{-n}$ with $n\ge 1$ we obtain
\begin{equation}\label{eqn_b_b1_n}
b_n - z_1 \sum_{i=0}^n b_{n-i} b_{i}^{(1)}+ 
\sum_{i=0}^{n+1} b_{n+1-i} b_{i}^{(1)}=0, \qquad n\ge 1. 
\end{equation}

At this stage we must consider two separate cases.  

\vspace{0.2cm}
\noindent
{\bf Case 1:}  $b_0=0$. In this case, equation \eqref{eqn_b_b1_1} is automatically satisfied, and from 
\eqref{eqn_b_b1_0} we find $b_0^{(1)}=a_1/b_1$. The remaining coefficients are obtained iteratively from
\[
b_n^{(1)}=-\frac{1}{b_1} \Bigg[ b_n +  
 \sum_{i=0}^{n-1} b_{i}^{(1)} \big( b_{n+1-i}- z_1 b_{n-i} \big) \Bigg], \qquad n\ge 1,
\]
which follows from \eqref{eqn_b_b1_n}.  

\vspace{0.2cm}
\noindent
{\bf Case 2:}  $b_0 \neq 0$.  
Then \eqref{eqn_b_b1_1} implies $b_0^{(1)}=0$. From \eqref{eqn_b_b1_0} and 
\eqref{eqn_b_b1_n} we obtain
\[
b_1^{(1)}=\frac{a_1}{b_0}-1,
\]
and for $n\ge 1$,
\[
b_{n+1}^{(1)}=-\frac{1}{b_0} \Bigg[ b_n +  
 \sum_{i=1}^{n} b_{i}^{(1)} \big( b_{n+1-i}- z_1 b_{n-i} \big) \Bigg].
\]

We observe that in Case~1 (when $b_0=0$), if we know the next $k$ coefficients $\{b_j\}_{1 \le j \le k}$ of the expansion of $R(z)$ at $z=\infty$, then we can determine $k$ coefficients 
$\{b_j^{(1)}\}_{0 \le j \le k-1}$ of the corresponding expansion of $R_1(z)$. In this case, the first coefficient $b_0^{(1)}$ is non-zero. In Case~2 (when $b_0 \neq 0$), if we know $k$ coefficients $\{b_j\}_{0 \le j \le k-1}$ of the expansion of $R(z)$ at $z=\infty$, then necessarily $b_0^{(1)}=0$, and we can compute the next $k$ coefficients 
$\{b_j^{(1)}\}_{1 \le j \le k}$ of the expansion of $R_1(z)$. This behavior is consistent with the simple example considered at the beginning of this section.  

Now we can describe the general algorithm for finding the rational function $R(z)$ satisfying \eqref{R_conditions_at_zj} and \eqref{eqn:R_at_infty2}. Define $\gamma_0=0$ and $\gamma_j=\xi_{j-1}$ for $j=1,2,\dots,n_{\infty}$, so that $R$ has the following asymptotic expansion as $z \to \infty$: 
\begin{equation}\label{R_infty_gamma}
R(z)=\sum_{j=0}^{n_{\infty}} \gamma_j z^{-j} + 
O(z^{-n_{\infty}-1}). 
\end{equation}
We now define $R_1$ by 
\begin{equation}\label{CF_transformation_finite2}
R(z)=\frac{a_1}{1+(z-z_1) R_1(z)}, 
\end{equation}
with interpolation values $a_j^{(1)}=R_1(z_j)$ for $2 \le j \le p$, computed using \eqref{R1_conditions_at_zi2}. Since $\gamma_0=0$, the function $R_1$ has the following expansion for large $z$: 
\begin{equation*}
R_1(z)=\sum_{j=0}^{n_{\infty}-1} \gamma^{(1)}_j z^{-j}+O(z^{-n_{\infty}}). 
\end{equation*}
Here $\gamma_j$ and $\gamma_j^{(1)}$ play the same role as $b_j$ and $b_j^{(1)}$ in the preceding discussion, and the coefficient $\gamma_0^{(1)}$ is non-zero. Note that if a rational function $R_1$ does not have a pole at $z_1$, the function $R$ constructed via 
\eqref{CF_transformation_finite2} will satisfy the interpolation condition $R(z_1)=a_1$. 

We then repeat the procedure, reducing the number of interpolation points at each step. Specifically, we define
\begin{equation*}
R_1(z)=\frac{a^{(1)}_2}{1+(z-z_2) R_2(z)}, 
\end{equation*}
and compute $a_j^{(2)}=R_2(z_j)$ for $j=3,4,\dots,p$ along with the coefficients of the expansion of $R_2(z)$ at $z=\infty$. Since $\gamma_0^{(1)}\neq 0$, the expansion of $R_2$ at infinity takes the form
\begin{equation*}
R_2(z)=\sum_{j=0}^{n_{\infty}} \gamma^{(2)}_j z^{-j}+O(z^{-n_{\infty}-1}), \qquad z\to \infty,
\end{equation*}
with $\gamma_0^{(2)}=0$.  The condition that $R_1$ does not have a pole at $z_1$ is equivalent to $R_2(z_1) \neq 1/(z_2-z_1)$. As long as this condition is satisfied, the function $R$ given by 
\eqref{CF_transformation_finite2} will satisfy the interpolation condition $R(z_1)=a_1$. In order for $R_1$ to satisfy the interpolation condition $R_1(z_2)=a_2^{(1)}$, the function $R_2$ should not have a pole at $z=z_2$. 

Repeating this procedure $p-2$ times, we define successively the functions $R_3, R_4, \dots, R_p$. Performing back-substitution (as in the illustrative example at the beginning of this section) yields the continued fraction representation of $R$: 
\begin{equation}\label{R_as_CF_ending_in_Rp}
R(z)=\cfrac{a_1}{1+\cfrac{(z-z_1) a^{(1)}_2}
{1+\cfrac{(z-z_2) a^{(2)}_3}{\ddots \cfrac{\;\;\dots\;\;}
{1+\cfrac{(z-z_{p-1}) a^{(p-1)}_p}{1+(z-z_p) R_p(z)}}}}}.
\end{equation}
As long as $R_p$ does not have a pole at $z_p$ and $R_j(z_{j-1}) \neq 1/(z_j-z_{j-1})$ for $j=2,3,\dots,p$, the rational function $R$ will satisfy the interpolation conditions $R(z_j)=a_j$ for $j=1,2,\dots,p$. 
Moreover, $R$ has the prescribed asymptotics \eqref{R_infty_gamma} if and only if $R_p$ satisfies  
\begin{equation}\label{eqn_R_p_at_infty}
R_p(z)=\sum_{j=0}^K \gamma^{(p)}_j z^{-j} + O(z^{-K-1}), \qquad z\to \infty,
\end{equation}
where the coefficients $\gamma_j^{(p)}$ are computed by the iterative procedure above, and 
\[
K=\begin{cases}
n_{\infty}, & \text{if $p$ is even}, \\
n_{\infty}-1, & \text{if $p$ is odd}.
\end{cases}
\]
Note that $K$ is always an even integer, since $p+n_{\infty}=2M$ is even.  
 
The last step of the algorithm is to express $R_p(z)$ (which must satisfy \eqref{eqn_R_p_at_infty}) as a continued fraction (see \cite[Section~4.2]{Baker1981}):  
\begin{equation}\label{eqn_R_p}
R_p(z)=\gamma_0^{(p)}+
\cfrac{d_1 z^{-1}}{1+\cfrac{d_2 z^{-1}}{\ddots 
\cfrac{\dots}{1+\cfrac{d_{K-1}z^{-1}}{1+d_{K} z^{-1}}}}}, 
\end{equation}
where $d_1:=\gamma_1^{(p)}$, and the remaining coefficients $\{d_j\}_{2\le j \le K}$ are computed by an iterative scheme similar to \eqref{CF_transformation_finite}. Specifically, we define $R_{p+1}(z)$ by 
\[
R_p(z)=\gamma_0^{(p)}+\frac{d_1}{z R_{p+1}(z)},  
\]
which is equivalent to 
\[
R_{p+1}(z) = \frac{d_1}{z \big(R_p(z)-\gamma_0^{(p)}\big)}.
\]
From \eqref{eqn_R_p_at_infty} and this relation, we compute the expansion of $R_{p+1}(z)$ as $z \to \infty$: 
\begin{equation*}
R_{p+1}(z)=1+\sum_{j=1}^{K-1} \gamma^{(p+1)}_j z^{-j} + O(z^{-K}).
\end{equation*}
We then set $d_2:=\gamma_1^{(p+1)}$ and repeat this step. Next, we define 
\[
R_{p+1}(z)=1+\frac{d_2}{z R_{p+2}(z)},
\]
compute the expansion of $R_{p+2}(z)$ at $z=\infty$, and set $d_3:=\gamma_1^{(p+2)}$. Continuing this procedure for $K$ steps yields the coefficients $\{d_j\}_{1\le j \le K}$ of the continued fraction \eqref{eqn_R_p}. Substituting $R_p$ from \eqref{eqn_R_p} into \eqref{R_as_CF_ending_in_Rp} then gives the desired rational function $R(z)$.  

There are some exceptional cases in which the algorithm may fail to produce a result, specifically when division by zero occurs. This can happen if, for some $1\le j \le p-1$, the function $R_j$ vanishes at one of the points $z_i$ with $j+1 \le i \le p$. Failure can also occur if $\gamma_0^{(l)}=\gamma_1^{(l)}=0$ for some $l$, since in that case the iteration cannot proceed further to compute $\gamma_j^{(l+1)}$. There are also rare cases where the algorithm produces a rational function $R$ with multiple poles or where $R$ fails to satisfy the interpolation conditions \eqref{R_conditions_at_zj}. As discussed above, this latter situation may occur when $R_j(z_{j-1}) =1/(z_j-z_{j-1})$ for some $2\le j \le p$ or if $R_p$ has a pole at $z_p$. In all of our numerical experiments, however, such situations did not arise, and we do not regard them as a major practical concern.  

We now claim that whenever the algorithm succeeds, the resulting rational function $R$ has degree $[M-1/M]$. When $n_{\infty}=2L$ is even, then $p=2M-n_{\infty}$ is also even and $\gamma_0^{(p)}=0$. By induction on $L$, one can show that in this case $R_p$ has degree $[L-1/L]$, and hence $R$ (constructed via \eqref{R_as_CF_ending_in_Rp}) has degree $[M-1/M]$. When $n_{\infty}=2L+1$ is odd, $p$ is also odd and $\gamma_0^{(p)}\neq 0$. A similar argument shows that $R_p$ has degree $[L/L]$, and therefore $R$ again has degree $[M-1/M]$. Thus, once we obtain the rational function $R$ and verify that it satisfies the interpolation conditions \eqref{R_conditions_at_zj}, Proposition~\ref{prop_unique} ensures that $R$ is the unique solution to the multi-point Pad\'e approximation problem.

\begin{remark}
\textnormal{
As noted earlier, the integer $K$ appearing in \eqref{eqn_R_p_at_infty} is always even. Writing $K=2L$, we may identify the rational function $R^*_p(z):=R_p(1/z)$ as the $[L/L]$ Pad\'e approximant at $z=0$ to the power series 
$\sum_{j\ge 0} \gamma^{(p)}_j z^j$.  There are several methods (besides the one described above) for computing the numerator and denominator coefficients of this Pad\'e approximant. For instance, one can determine them by solving a linear system, as explained in \cite[Section~2.3]{Baker1981}. An efficient alternative that also uses continued fractions is the Q.D. algorithm (see \cite[Section~4.3]{Baker1981}); this method may be preferable when $n_{\infty}$ is large. In our examples $n_{\infty}$ is relatively small, so we used the continued fraction approach described above, which is straightforward to implement.}
\end{remark}

\section{Numerical results}\label{section_Numerical_results}

We now discuss the implementation and performance of our algorithm.  
Implementation requires high precision: we used a working precision of 100 decimal digits to compute the coefficients $c_j$ and $\lambda_j$. High precision is necessary due to the iterative nature of the continued fraction algorithm (where loss of precision may accumulate) and the need to compute roots of polynomials of relatively high degree (in some of our examples, up to degree~60).  

We wrote the code in Fortran and employed D.~H.~Bailey’s highly efficient MPFUN2020 \cite{Bailey_2020} arbitrary-precision package. Our code is available for download at \url{https://kuznetsovmath.ca/} or at  \href{https://github.com/Alexey-Kuznetsov-math/exponential_sum_approximations}{github}. The module {\texttt{polynomial\_module.f90}} defines two derived types: a polynomial and a rational function. A polynomial is specified by a sequence of coefficients of type {\texttt{mp\_real}} or {\texttt{mp\_complex}} (multiple-precision types provided by the MPFUN2020 package). A rational function is represented as a pair of polynomials corresponding to its numerator and denominator. This module includes subroutines for basic operations with polynomials and rational functions (addition, subtraction, multiplication, and division), as well as routines for computing roots of polynomials and partial fraction decompositions of rational functions. Polynomial roots are computed using the Ehrlich–Aberth method \cite{Aberth_1973,Ehrlich_1967}.  

The algorithm requires the values $F(z_j)$ of the Laplace transform of $f$. If $F(z)$ is not known in closed form, we evaluate $F(z_j)$ numerically by applying the double-exponential quadrature of Takahasi–Mori \cite{Takahasi_1974} to the integral in \eqref{def:F_Laplace_transform}: 
\begin{equation}\label{double_exp_quadrature}
F(z)=\int_0^{\infty} f(x) e^{-z x} \, \d x 
\approx h \sum_{n \in \mathbb{Z}}
x_n f(x_n) e^{-z x_n} \big(1+e^{-nh}\big),
\end{equation}
where $h>0$ is a small step size and $x_n:=\exp(nh-\exp(-nh))$.  
Formula \eqref{double_exp_quadrature} is obtained by applying the trapezoidal rule to the integral after the change of variables $x=\exp(u-\exp(-u))$. The infinite sum converges very rapidly and is truncated once its terms fall below the working precision.  

Our implementation of the exponential sum approximation algorithm is reasonably fast. For example, finding the coefficients $c_j$ and $\lambda_j$ for the exponential sum approximation to the hockey stick function (see Section~\ref{subsection_hockeystick}) with $M=30$ ($M=60$) terms took about $0.1$ seconds (respectively, $1$ second) on a standard laptop (Intel Core i5-1250P processor with 16GB of RAM). This example is relatively simple since the Laplace transform is known explicitly. When $F(z_j)$ must be computed numerically, the runtime increases substantially. In the lognormal example of Section~\ref{subsection_Lognormal}, computing the coefficients $c_j$ and $\lambda_j$ for $M=30$ took $1.5$ seconds. Fortunately, the computation of $F(z_j)$ for different values of $j$ can be fully parallelized. With OpenMP parallelization (compiled using the {\texttt{-fopenmp}} flag), the same example with $M=30$ took only $0.4$ seconds instead of $1.5$ seconds in serial execution.

Our algorithm depends on four parameters: two positive integers $M$ and $n_{\infty}$ (with the constraint that $p=2M-n_{\infty} \ge 2$), and two real numbers $A\ge 0$ and $B>0$. Based on extensive experiments, we can offer the following empirical guidelines for choosing these parameters:  
\begin{itemize}
\item Larger values of $M$ generally yield better approximations, corresponding to more terms in the exponential sum \eqref{def:phi}.  
\item Increasing $n_{\infty}$ improves the accuracy of $f(x)-\phi(x)$ for small values of $x$, since $f(x)-\phi(x)=O(x^{n_{\infty}})$ as $x \to 0^+$.  
\item Increasing $A$ has a similar effect, as it improves the approximation $F(z)-R(z)$ for large $z$, and Watson’s Lemma relates the large-$z$ behavior of the Laplace transform to the small-$x$ behavior of the original function. 
\item In some cases (such as the hockey stick example of Section~\ref{subsection_hockeystick}), the algorithm produced exponential sums with very large coefficients $c_j$. Increasing $B$ alleviated this issue and reduced the values of $c_j$ to acceptable levels.  
\end{itemize}

We now turn to the results of various numerical experiments. For each example considered in the remainder of this section, the parameters $c_j$ and $\lambda_j$ of all exponential sum approximations can be downloaded at \url{https://kuznetsovmath.ca/}.

\subsection{Approximating the Gaussian function}\label{subsection_Gaussian}

As our first example, we consider the Gaussian function $f(x)=\exp(-x^2)$.  
Approximations of this function by exponential sums of the form \eqref{def:phi} are widely used, for instance in the fast Gauss transform \cite{Jiang2022,Kunis_2006}. Since the Laplace transform of $f$ is not known in closed form, we computed the values of $F(z_j)$ numerically using the double-exponential quadrature \eqref{double_exp_quadrature}.

For our first experiment, we fixed $M=12$ and $n_{\infty}=2$ and used the algorithm described on page \pageref{page_algorithm} to  compute exponential sums $\phi=\phi_{A,B}$ for a range of values of $A$ and $B$. In Figure \ref{fig13} we show the contour plot of $\log_{10} \| f -\phi_{A,B} \|_{\infty}$. We see that the best approximation (in the $L_{\infty}$ norm) is obtained near $A=4.5$ and $B=10$. The graph of $f-\phi$ (where $\phi$ is computed with these values of $A$ and $B$) is shown  in Figure \ref{fig_2d}. 

Next, we repeated the same procedure. We fixed $n_{\infty}=2$ and for each $M \in \{12,18,24\}$ we tested a range of parameters $A$ and $B$ (with step size $0.5$) and selected those values of $A$ and $B$ yielding the smallest $L_1$ or $L_{\infty}$ error between $f$ and $\phi$.  The results are shown in Figure~\ref{fig2}.  
The top row of graphs (\ref{fig_2a}–\ref{fig_2c}) shows approximations optimized for the smallest $L_1$ error $\|f-\phi\|_1$, while the bottom row (\ref{fig_2d}–\ref{fig_2f}) shows those optimized for the smallest $L_{\infty}$ error. For the approximation in Figure~\ref{fig_2c}, we obtained  
\[
\| f - \phi\|_1 \approx 6.4\times 10^{-23}, 
\qquad 
\| f - \phi\|_{\infty} \approx 5.0\times 10^{-23}.
\]
For the approximation in Figure~\ref{fig_2f}, the errors were  
\[
\| f - \phi\|_1 \approx 8.5\times 10^{-23}, 
\qquad 
\| f - \phi\|_{\infty} \approx 2.4\times 10^{-23}.
\]
We observe a slight improvement (by about a factor of two in the $L_{\infty}$ error) when optimizing the parameters $A$ and $B$. The results in Figure~\ref{fig2} also suggest that both $L_1$ and $L_{\infty}$ errors decay exponentially with $M$.

\begin{figure}[t]
\centering
\includegraphics[height =7.5cm]{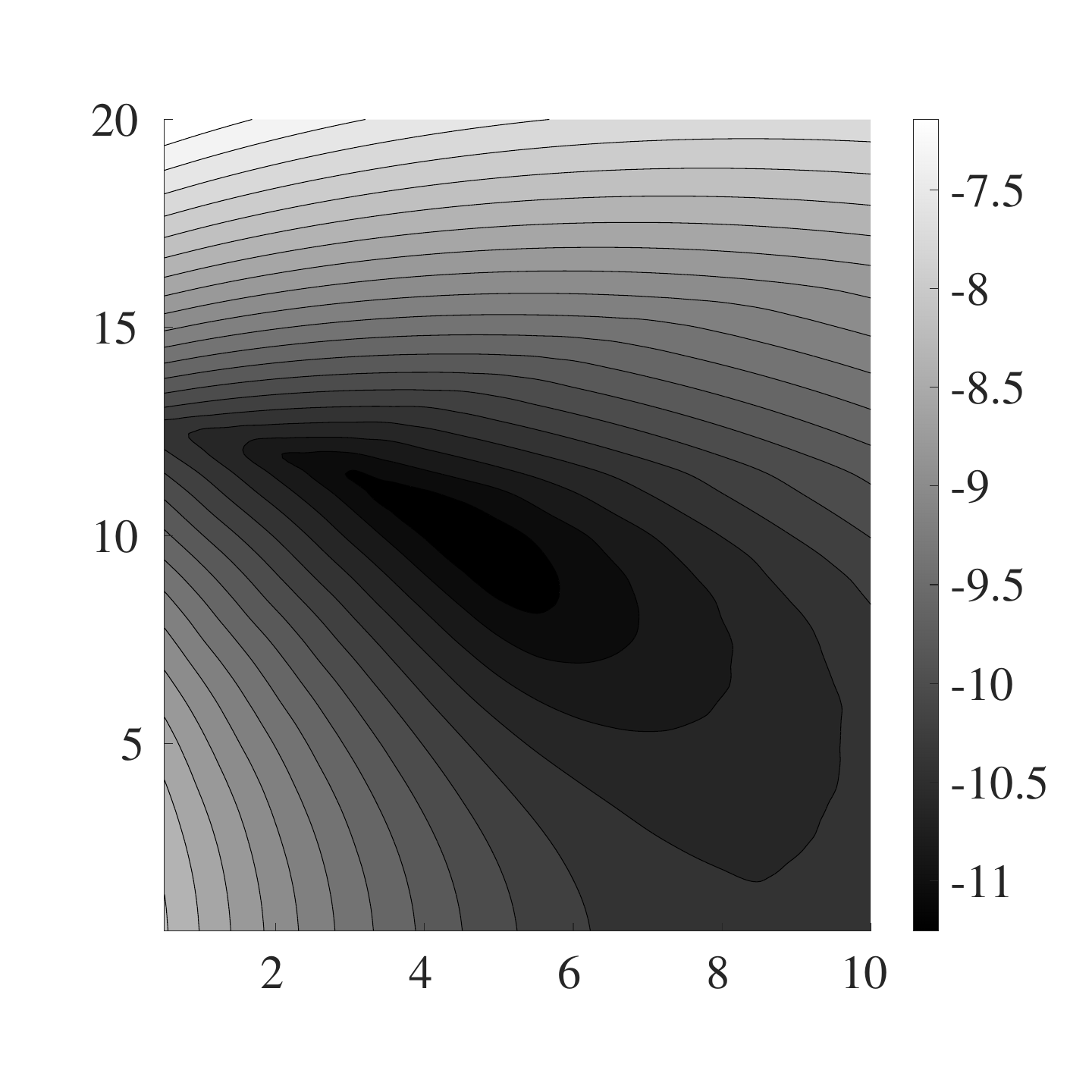} 
\caption{The contour plot of $\log_{10} \| f-\phi \|_{\infty}$ as a function of $A$ (on the $x$-axis) and $B$ (on the $y$-axis). Here $M=12$ and $n_{\infty}=2$. } 
\label{fig13}
\end{figure}
  
\begin{figure}[t]
\centering
\subfloat[][$M=12$, $A=3.5$, $B=10.5$]{\label{fig_2a}\includegraphics[height =5.5cm]{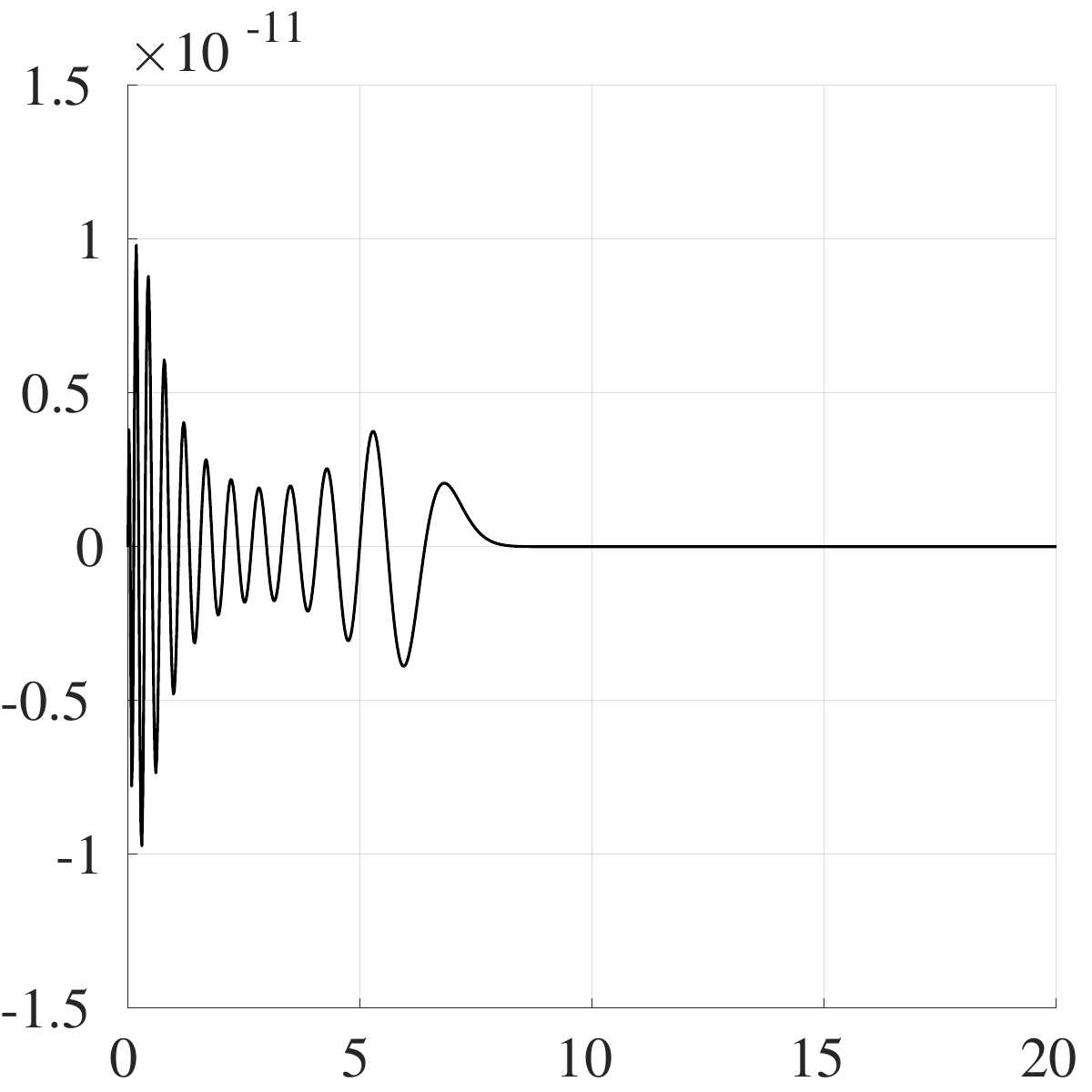}}
\hspace{0.15cm}
\subfloat[][$M=18$, $A=5$, $B=13.5$]{\label{fig_2b}\includegraphics[height =5.5cm]{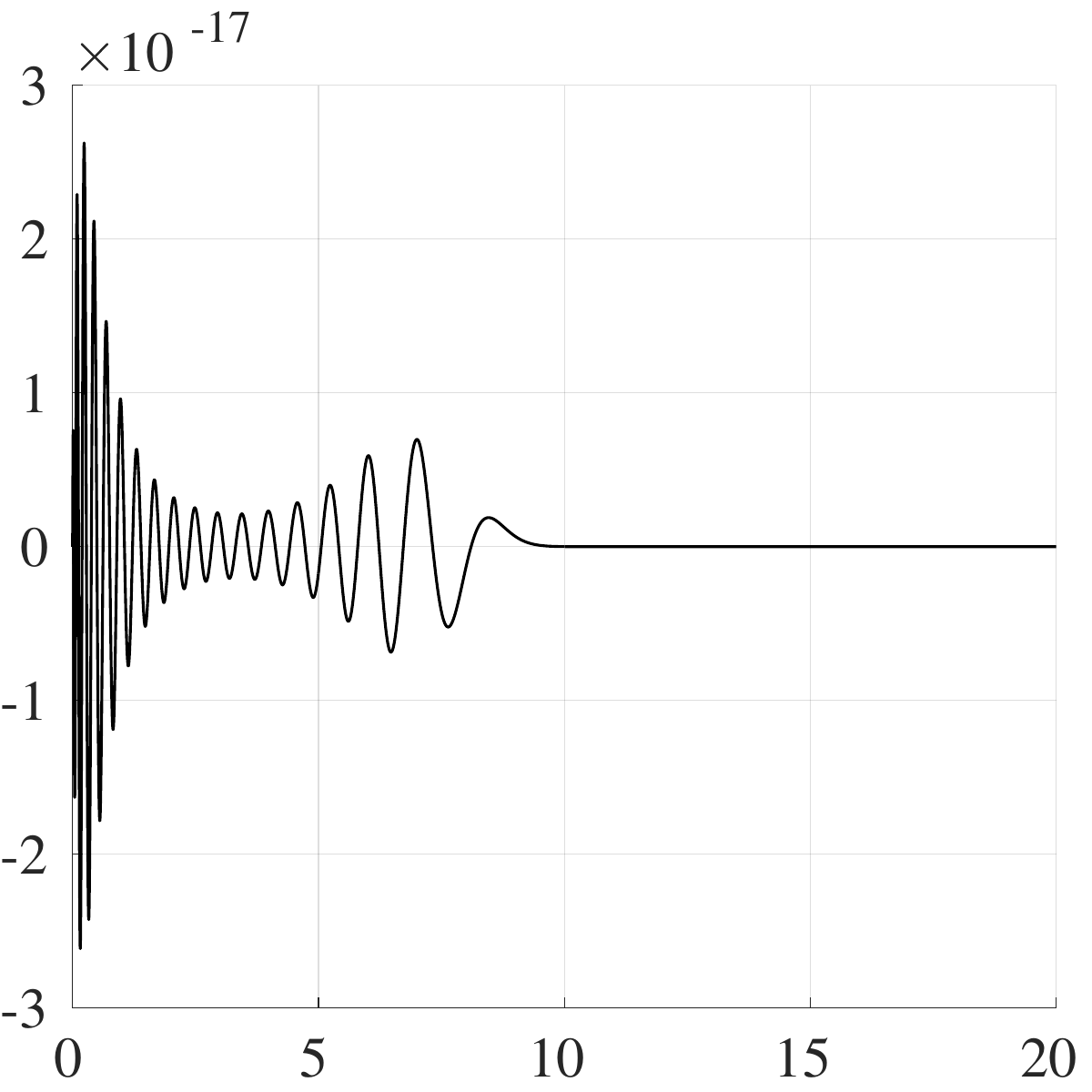}} 
\hspace{0.15cm}
\subfloat[][$M=24$, $A=6.5$, $B=16$]{\label{fig_2c}\includegraphics[height =5.5cm]{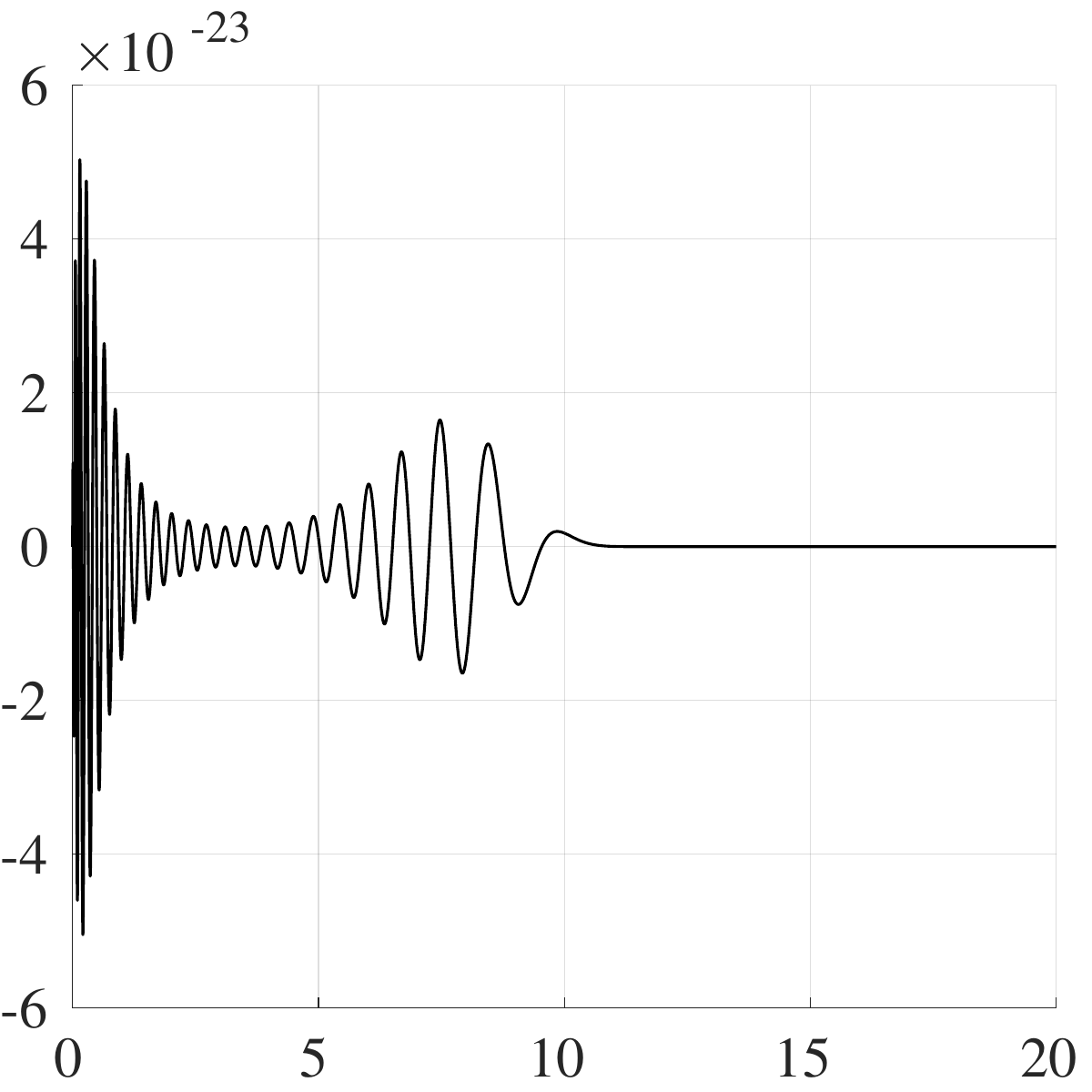}} \\
\hspace{0.15cm}
\subfloat[][$M=12$, $A=4.5$, $B=10$]{\label{fig_2d}\includegraphics[height =5.5cm]{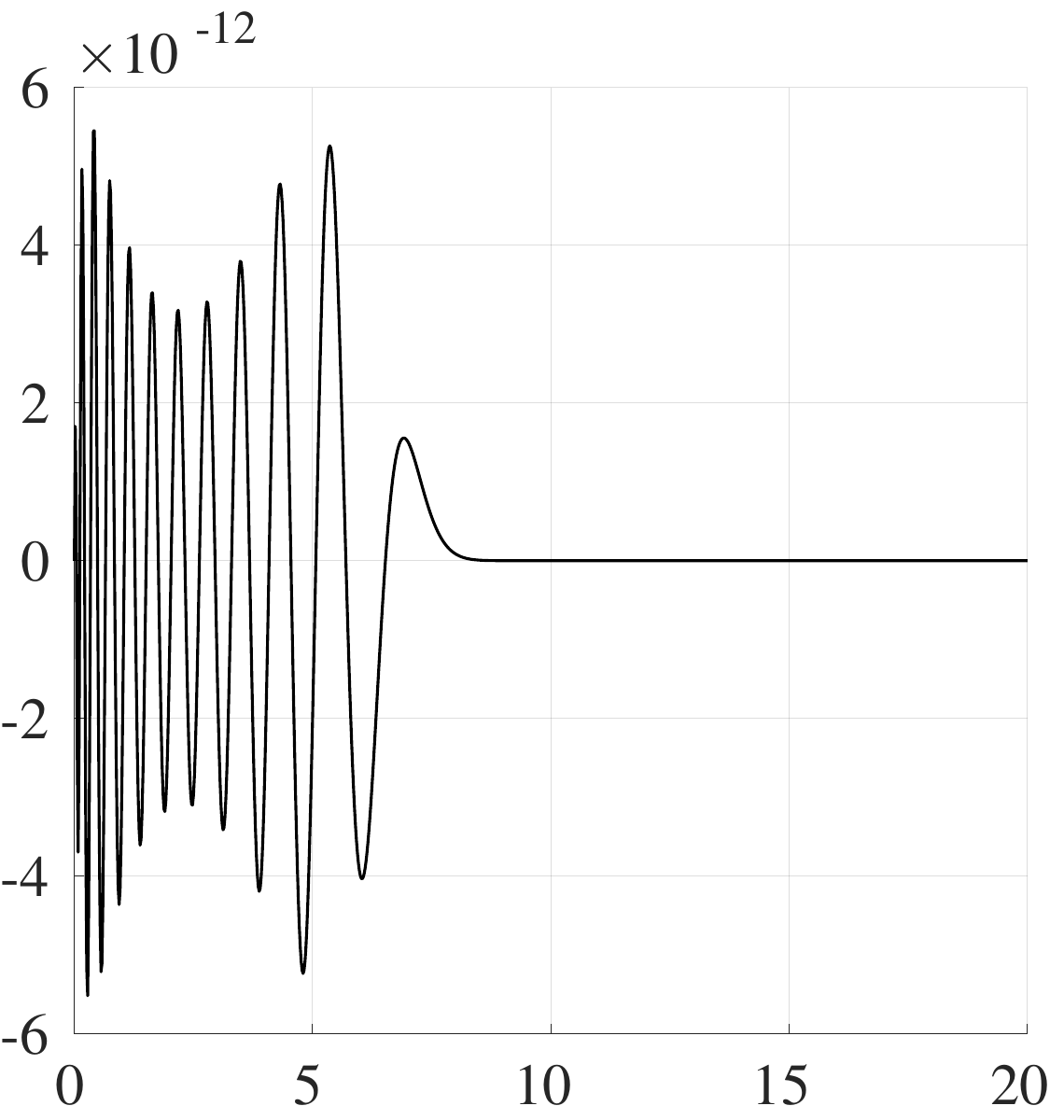}} 
\hspace{0.15cm}
\subfloat[][$M=18$, $A=6.5$, $B=11$]{\label{fig_2e}\includegraphics[height =5.5cm]{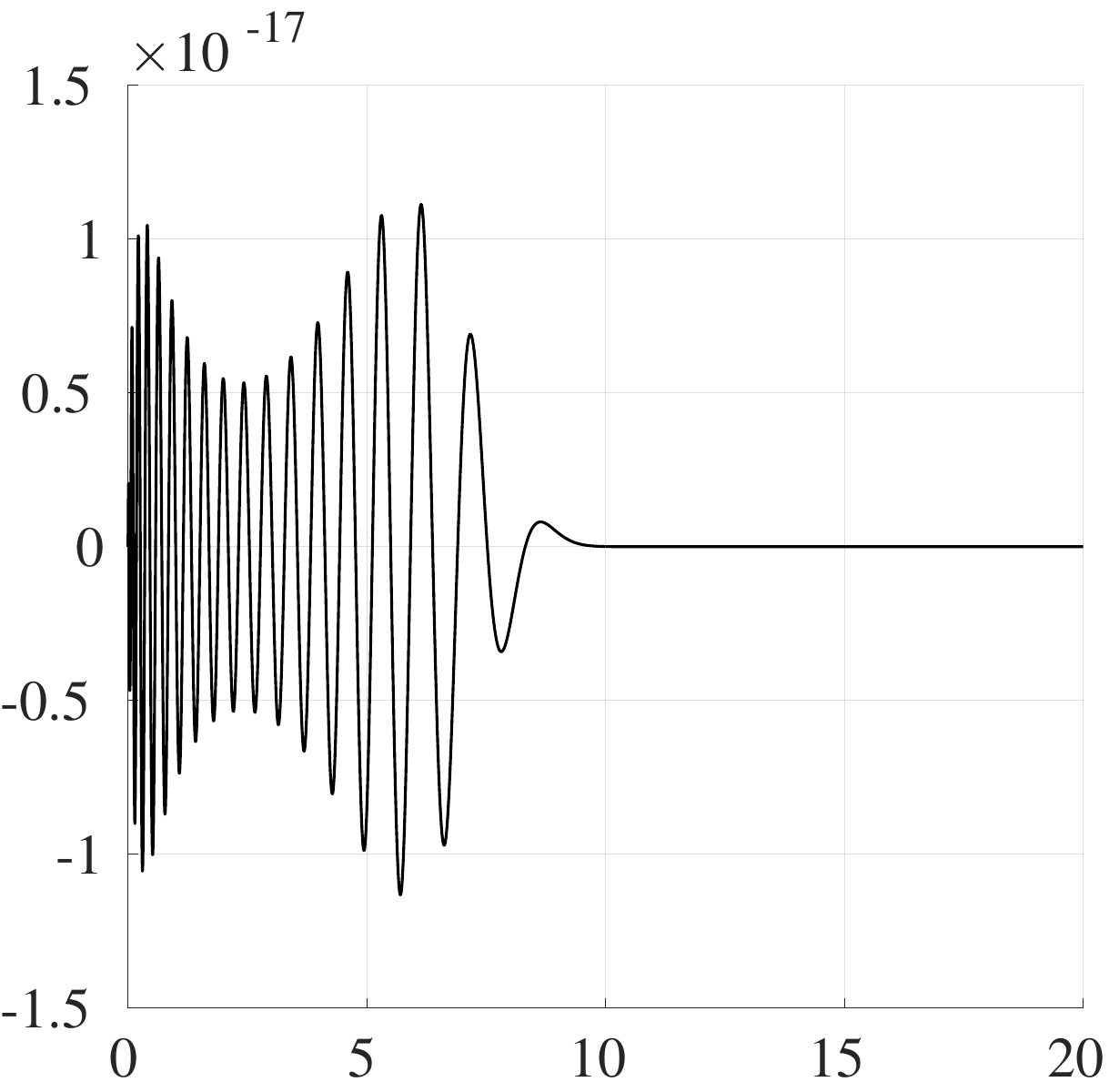}} 
\hspace{0.25cm}
\subfloat[][$M=24$, $A=8$, $B=14.5$]{\label{fig_2f}\includegraphics[height =5.5cm]{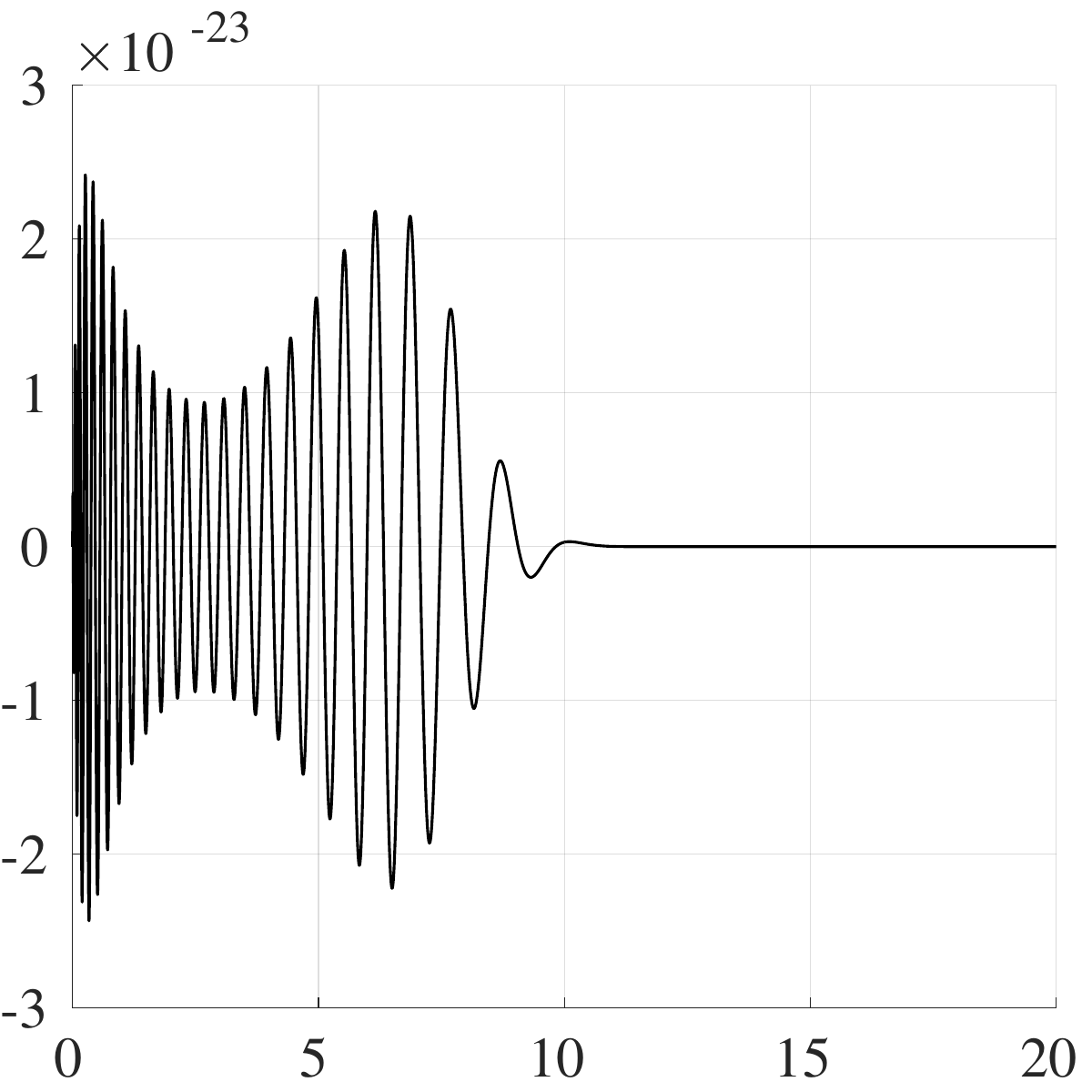}}   
\caption{The graphs of $\exp(-x^2)-\phi(x)$ for different values of $M$, $A$, and $B$. In each case we set $n_{\infty}=2$.} 
\label{fig2}
\end{figure}

\subsection{Approximations to the gamma function and the Barnes $G$-function}
\label{subsection_Gamma}

Let $G(z)$ denote the Barnes $G$-function and $\Gamma(z)$ denote the gamma function. These functions are defined via equations (5.2.1) and (5.17.3) in \cite{NIST} and satisfy $G(1)=\Gamma(1)=1$ and 
\[
\Gamma(z+1)=z \Gamma(z), \;\;\; G(z+1)=\Gamma(z) G(z), \;\;\;  z\in \c.
\]

The problem of approximating the logarithms of $\Gamma(z)$ and $G(z)$ was addressed in \cite{Kuznetsov_2022} with the help of exponential sum approximations.  Let $\phi$ be an exponential sum of the form \eqref{def:phi} that approximates the function
\begin{equation}\label{def:f(x)}
f(x):=\frac{e^{-x}}{x^3} \left[\frac{1}{2} \coth\left(\frac{x}{2}\right)- \frac{1}{x}-\frac{x}{12} \right], \;\;\; x>0. 
\end{equation}
Note that $f$ is smooth, decays exponentially as $x\to +\infty$, and has a removable singularity at $x=0$ (in fact, it is analytic in the disk $|x|<4\pi$). Let $\Phi$ be the rational function
\begin{equation*}
 \Phi(z):=\int_0^{\infty} e^{-z x} x \phi(x) \, \d x=
  \sum_{j=1}^M \frac{c_j}{(z+\lambda_j)^{2}}. 
\end{equation*}
We define approximations to $\ln(\Gamma(z))$ and $\ln(G(z))$ by
\begin{align}
\label{eq:ln_Gamma_approx}
\ln(\widehat \Gamma(z))&:=\left(z-\frac12 \right)\ln(z) - z+\frac12 \ln(2\pi)  + \frac{1}{12z}-\Phi'(z-1), \\
\label{eq:ln_G_approx}
\ln(\widehat G(z))&:=\left(\frac{z^2}{2}-z+\frac{5}{12}\right)\ln(z)
-\frac{3}{4}z^2
+\frac{1}{2}\ln(2\pi)(z-1)+z \nonumber\\
&+\frac{1}{12}-\ln({\mathcal A})-\frac{1}{12z}+\Phi(z-1)-(z-1)\Phi'(z-1),
\end{align}
where ${\mathcal A}=1.282427\dots$ is the Glaisher--Kinkelin constant.  
In \cite[Proposition 2]{Kuznetsov_2022} it was shown that for all $z$ in the half-plane $\re(z) \ge 3/2$ 
\begin{equation}\label{gamma_G_approximations}
 \big|\ln(\Gamma(z))-\ln(\widehat \Gamma(z))\big| \le  \epsilon_1, \;\;\; 
 \big|\ln(G(z))-\ln(\widehat G(z))\big| \le  \epsilon_2, 
\end{equation}
where $\epsilon_i=\Vert \eta_i \Vert_{\infty}=\sup\{ |\eta_i(x)| : x\ge 0\}$
and the functions $\eta_i$ are defined as follows: 
\[
 \eta_1(x):=2x^2 \big(f(x)-\phi(x)\big), \;\;\; 
 \eta_2(x):=6x\big(f(x)-\phi(x)\big)+2x^2\big(f'(x)- \phi'(x)\big).
\]

Thus, if we can find an exponential sum approximation $\phi$ to $f$ (given by \eqref{def:f(x)}) for which $\epsilon_i$ are small, then we obtain easily computable approximations to the logarithms of the gamma function and the Barnes $G$-function in the half-plane $\re(z)\ge 3/2$. Using functional equations and reflection formulas, these approximations can then be extended to compute the logarithms of $\Gamma(z)$ and $G(z)$ throughout the cut plane $\c \setminus (-\infty,0]$; see \cite{Kuznetsov_2022} for details.  

\begin{figure}[t]
\centering
\subfloat[][$M=12$, $n_{\infty}=2$, $A=2$, $B=3.5$]{\label{fig_3a}\includegraphics[height =6.5cm]{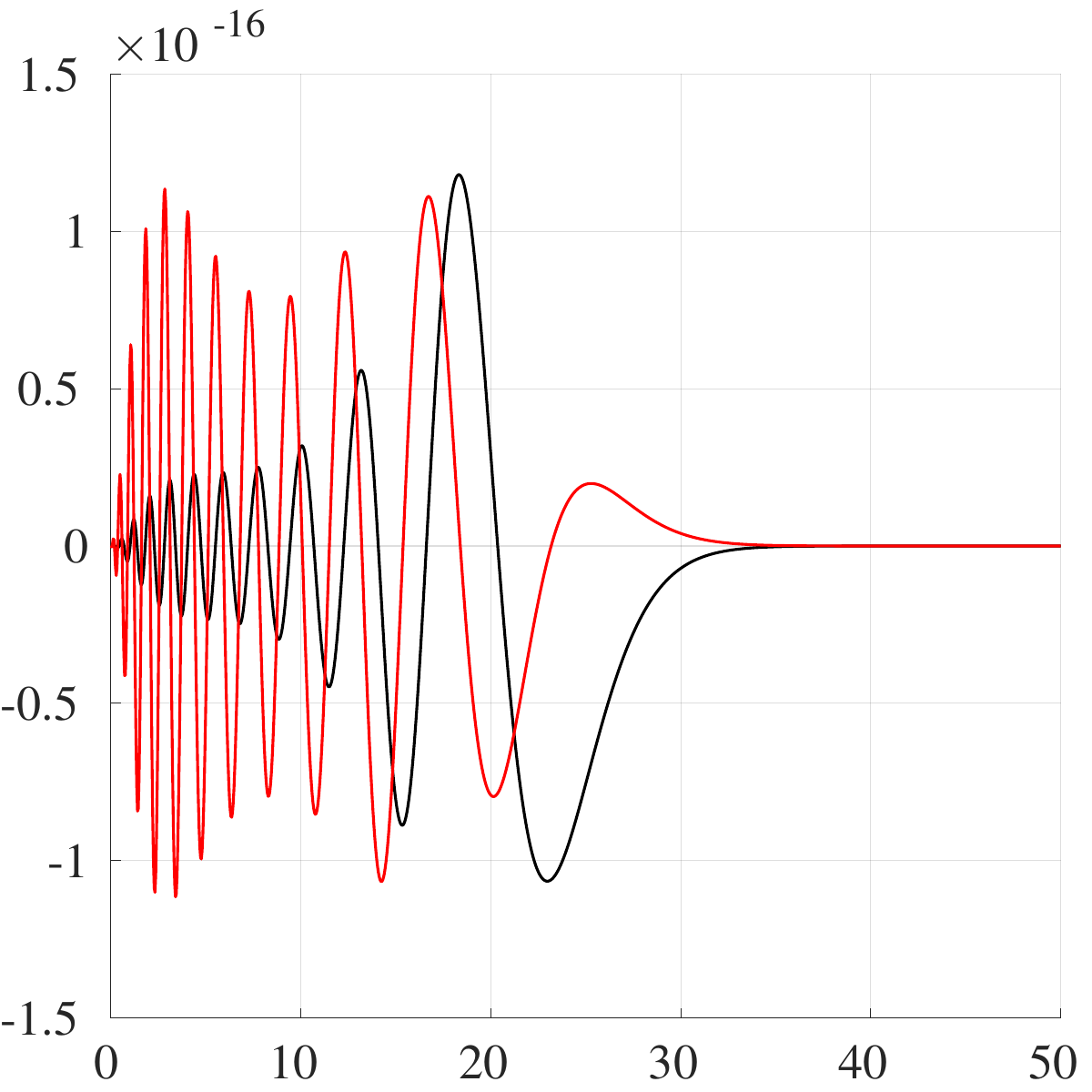}} 
\hspace{0.75cm}
\subfloat[][$M=30$, $n_{\infty}=4$, $A=4.25$, $B=5$]{\label{fig_3b}\includegraphics[height =6.5cm]{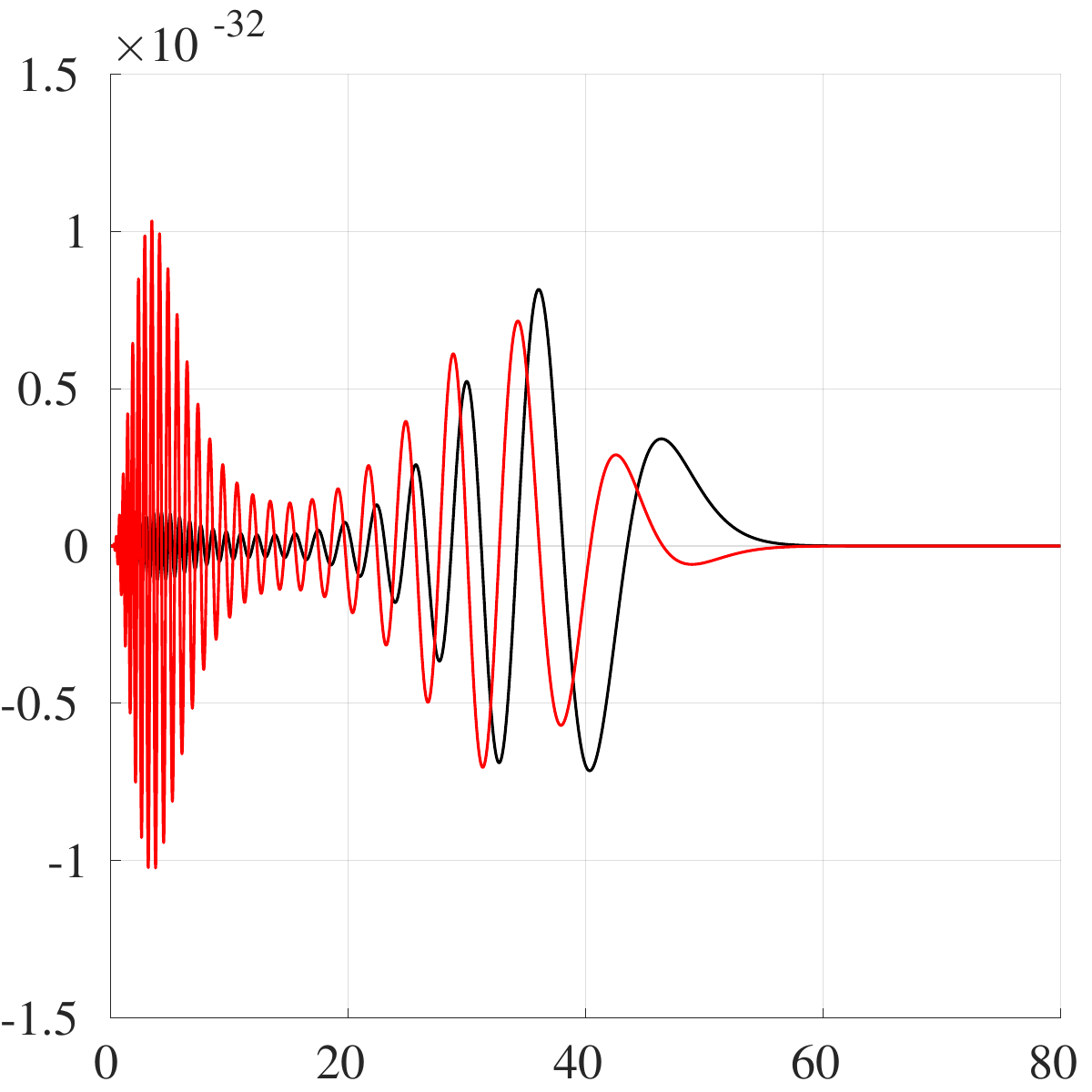}}   
\caption{The graphs of $\eta_1(x)$ (black) and $\eta_2(x)$ (red) for different values of $M$, $n_{\infty}$, $A$, and $B$.} 
\label{fig3}
\end{figure} 

The paper \cite{Kuznetsov_2022} employed a two-point Pad\'e approximation (interpolating at $0$ and $\infty$) to obtain two exponential sum approximations to $f$. The first approximation had $M=15$ terms and produced $\epsilon_1 \approx 9\times 10^{-17}$ and $\epsilon_2 \approx 2.8\times 10^{-16}$. According to \eqref{gamma_G_approximations}, this guarantees double-precision accuracy for computing $\ln(\Gamma(z))$ and $\ln(G(z))$ via their approximations in \eqref{eq:ln_Gamma_approx} and 
\eqref{eq:ln_G_approx} in the half-plane $\re(z) \ge 3/2$.  
We applied the algorithm presented in this paper (for which we had to compute the values of $F(z_j)$ via the double-exponential quadrature \eqref{double_exp_quadrature}), and we found an exponential sum $\phi$ with only $M=12$ terms, yielding $\epsilon_1, \epsilon_2 \approx 1.2 \times 10^{-16}$ (see Figure~\ref{fig_3a}). Thus, we achieve the same double-precision accuracy in approximating $\ln(\Gamma(z))$ and $\ln(G(z))$ as in \cite{Kuznetsov_2022} with 20\% fewer terms, which is significant since it leads to faster evaluation of $\ln(\Gamma(z))$ and $\ln(G(z))$.  
The second exponential sum approximation that was obtained in \cite{Kuznetsov_2022} had $M=45$ terms and produced $\epsilon_1 \approx 5 \times 10^{-32}$ and $\epsilon_2 \approx 2.9\times 10^{-31}$, providing nearly quadruple-precision accuracy for computing $\ln(\Gamma(z))$ and $\ln(G(z))$ via \eqref{eq:ln_Gamma_approx} and \eqref{eq:ln_G_approx}. Using our new method, we obtained an exponential sum $\phi$ with $M=30$ terms and $\epsilon_1, \epsilon_2 \approx  10^{-32}$ (see Figure~\ref{fig_3b}), thus achieving the same accuracy with 33\% fewer terms. Using these exponential sum approximations we wrote MATLAB/Python/Fortran functions for computing the logarithm of the gamma function in the entire complex plane, accurate to double and quadruple precision -- these can be downloaded at \href{https://github.com/Alexey-Kuznetsov-math}{github.com/Alexey-Kuznetsov-math}. MATLAB code for computing the logarithm of the Barnes $G$-function in the entire complex plane can be downloaded \href{https://www.mathworks.com/matlabcentral/fileexchange/182574-ln_barnes_g}{here}.

\subsection{Approximating the probability density function of the Gompertz--Makeham distribution}\label{subsection_Gompertz}

Consider the probability density function $f$ of the Gompertz-Makeham law of mortality:
\begin{equation*}
f(x)=(a+bc^{x_0+x})\exp\!\left(-ax-\frac{b}{\ln(c)}c^{x_0} (c^x-1)\right), \qquad x\ge 0. 
\end{equation*}
A particular instance of this distribution with parameters $x_0=65$, $a=0.0007$, $b=0.00005$, and $c=10^{0.04}$ was studied in \cite{Feng_2017, Feng_2019}, where the authors used the Beylkin-Monz\'on method \cite{Beylkin_2005} to approximate $f$ by exponential sums. Their approximation employed $M=15$ terms and achieved an $L_{\infty}$ error of approximately $\|f-\phi\|_{\infty} \approx 10^{-6}$.  

Using our algorithm with $M=14$ terms, we obtained an approximation with a smaller error, $\|f-\phi\|_{\infty} \approx 1.5\times 10^{-7}$.  Furthermore, with $M=28$ terms we achieved $\|f-\phi\|_{\infty} \approx 2.2\times 10^{-12}$.   
Since the Laplace transform of $f$ is not known in closed form, we computed the values of $F(z_j)$ via the double-exponential quadrature \eqref{double_exp_quadrature}. The density function $f$ and the corresponding errors $f-\phi$ are shown in Figure~\ref{fig4}.

\begin{figure}[t]
\centering
\subfloat[][Gompertz--Makeham pdf $f(x)$]{\label{fig_4a}\includegraphics[width =5.45cm]{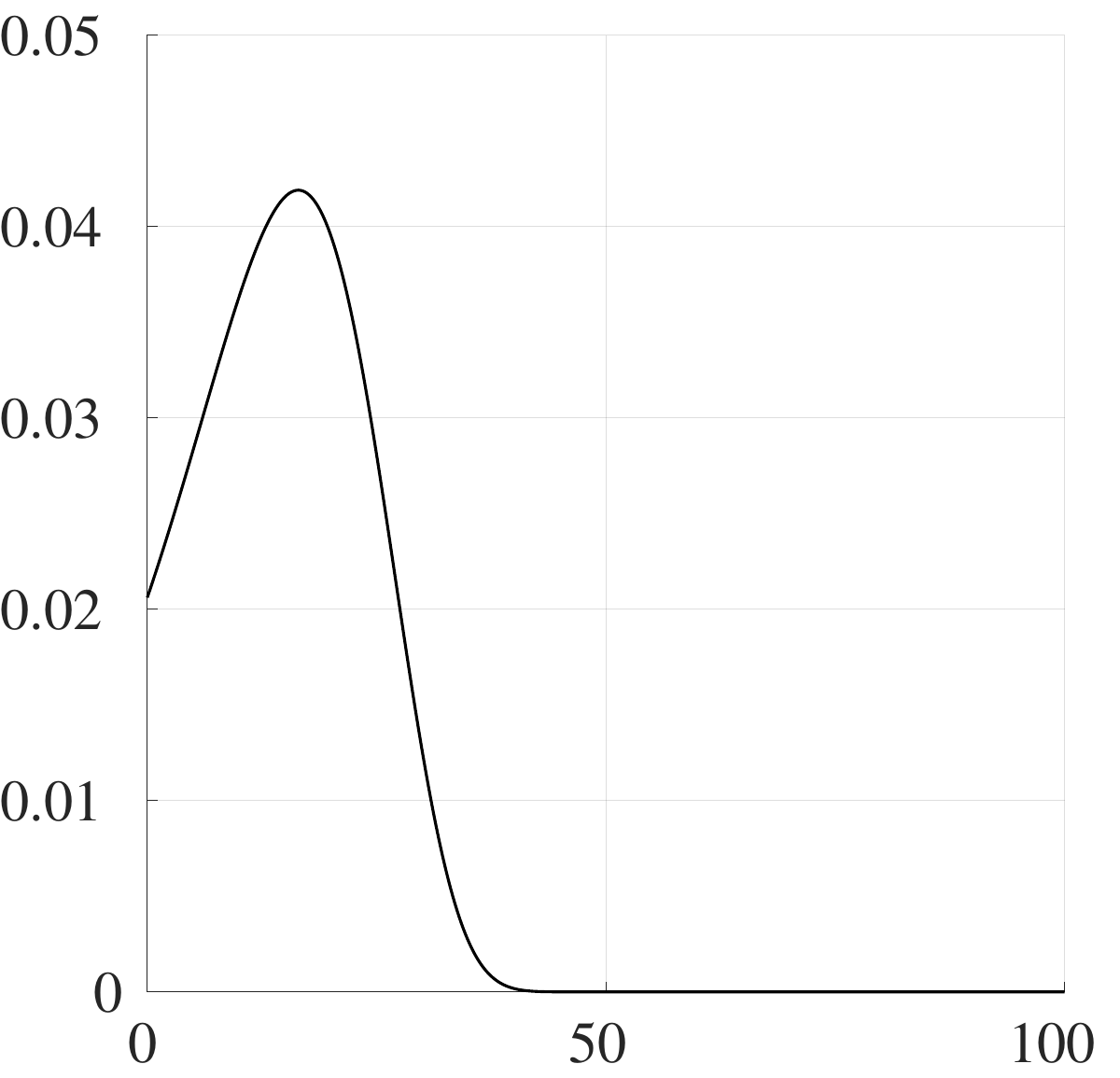}} 
\hspace{0.15cm}
\subfloat[][$M=14$, $n_{\infty}=2$, $A=0.1$, $B=0.9$]{\label{fig_4b}\includegraphics[width =5.5cm]{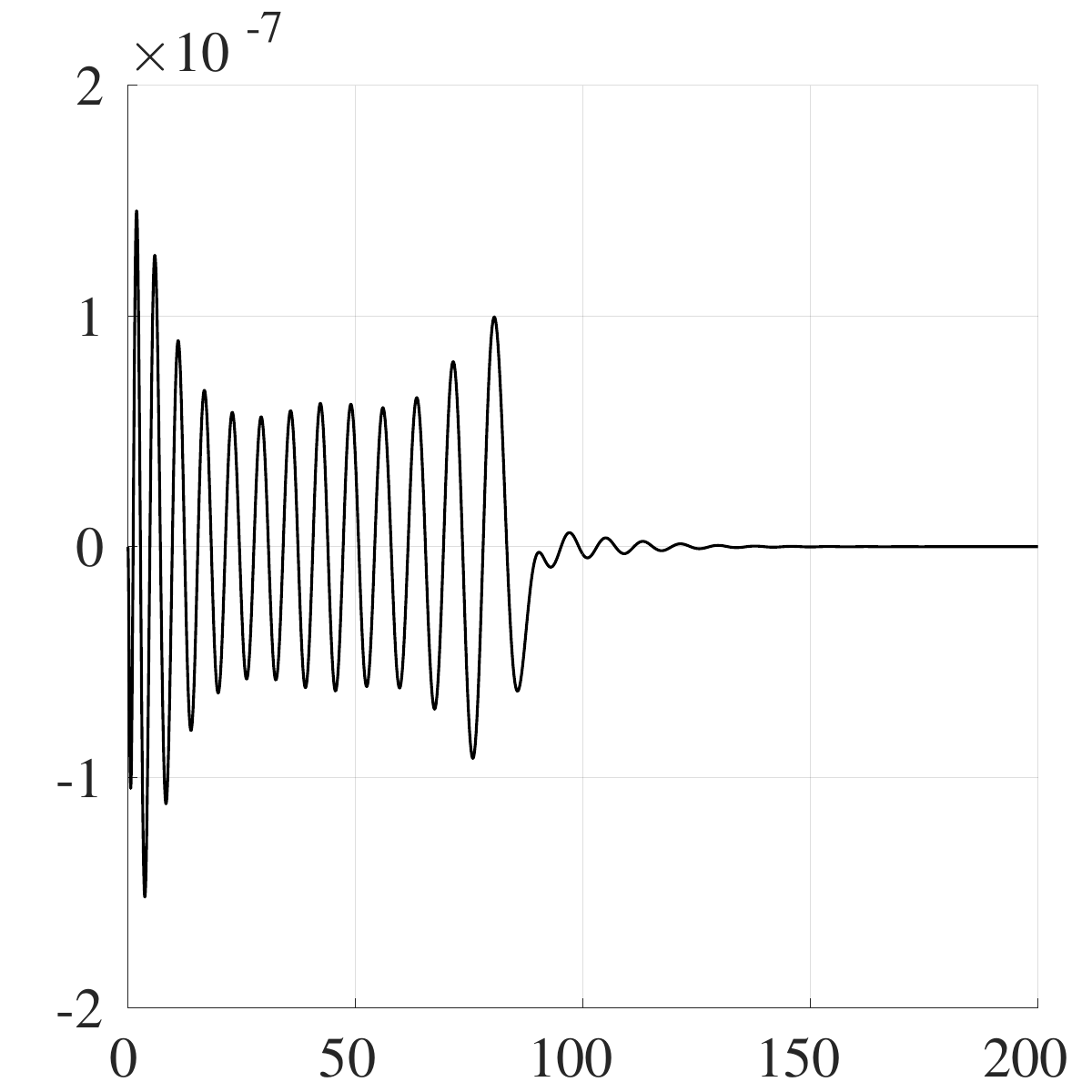}}   
\hspace{0.15cm}
\subfloat[][$M=28$, $n_{\infty}=4$, $A=0.12$, $B=1.65$]{\label{fig_4c}\includegraphics[width =5.5cm]{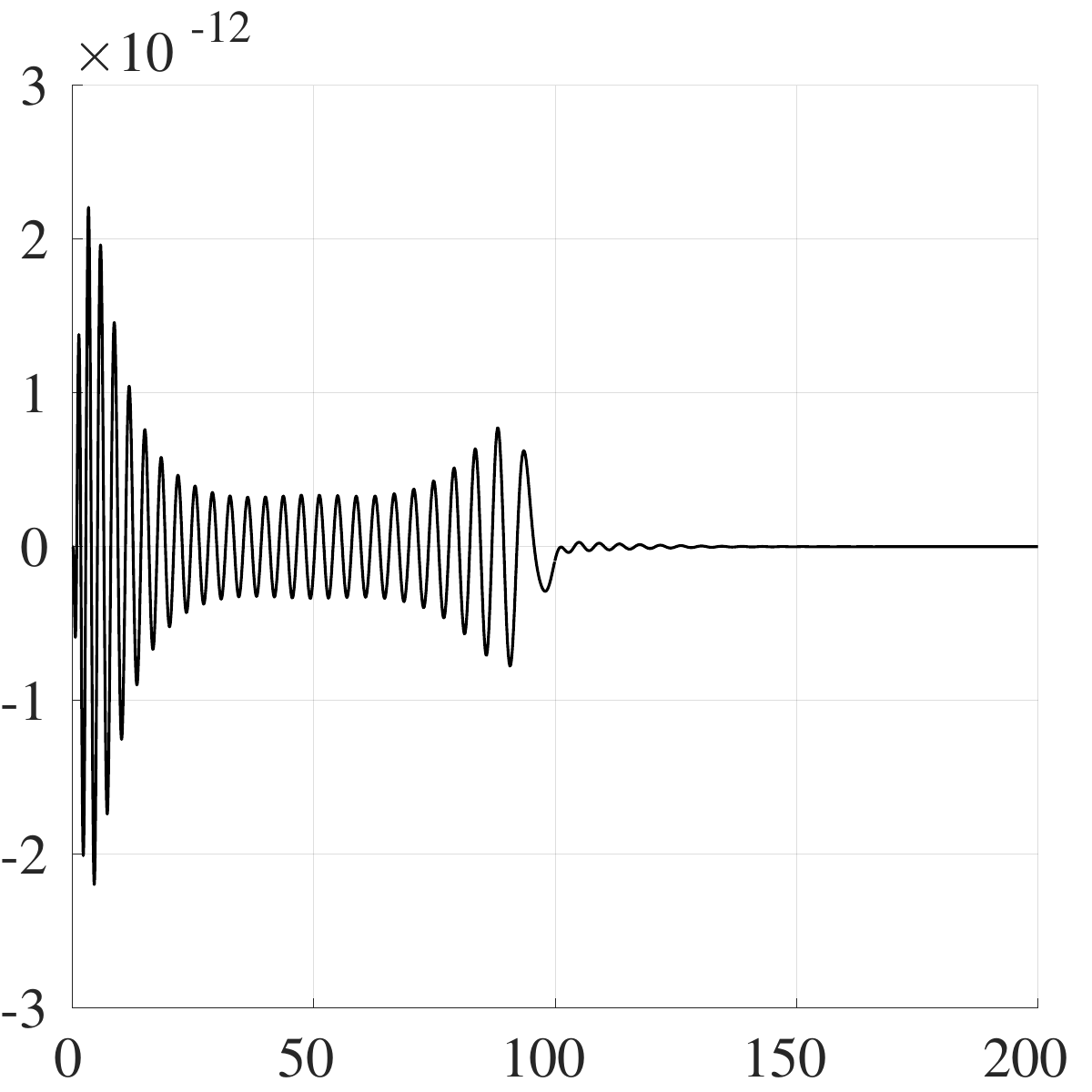}}  
\caption{Approximating the Gompertz--Makeham probability density function. Subplots (b) and (c) show the error $f(x)-\phi(x)$ for exponential sum approximations with $14$ and $28$ terms, respectively.} 
\label{fig4}
\end{figure}

Thus, in this example, our method improves upon the Beylkin-Monz\'on approximation, achieving smaller errors with fewer terms. This improvement stems from the fact that our algorithm allows for better control of the approximation error $f(x)-\phi(x)$ for small values of $x$, through the parameter $n_{\infty}$.

\subsection{Approximating the density of the lognormal distribution}\label{subsection_Lognormal}

The probability density function of a lognormal distribution is given by
\[
g_{\sigma}(x):=\frac{1}{x \sigma \sqrt{2\pi}} 
\exp\!\left(-\frac{1}{2\sigma^2 }\ln(x)^2\right), \qquad x>0, 
\]
where $\sigma>0$ is a parameter. Graphs of $g_{\sigma}(x)$ for $\sigma \in \{0.5, 1, 1.5\}$ are shown in Figure~\ref{fig5}.  

A key property of this density is that all derivatives $g^{(j)}_{\sigma}(x)$ vanish as $x\to 0^+$. In particular, our continued fraction algorithm cannot be applied directly since $g_{\sigma}(0+)=0$. We address this difficulty by a simple modification: instead of approximating $g_{\sigma}(x)$, we approximate the function  
\[
f(x)=\int_x^{\infty} g_{\sigma}(y)\, \d y,
\]
by an exponential sum $\phi(x)$ of the form \eqref{def:phi}, and then define  
\begin{equation}\label{tilde_phi}
\widetilde \phi(x)=-\phi'(x)=\sum_{j=1}^M c_j \lambda_j e^{-\lambda_j x}
\end{equation}
as an exponential sum approximation to $g_{\sigma}(x)=-f'(x)$.  

Since $f(0)=1$ and $f'(x)=-g_{\sigma}(x)$, the coefficients $\xi_j$ in \eqref{eqn:f_at_zero} are given by  
\begin{equation}\label{xi_j_lognormal}
\xi_0=1, \qquad \xi_j=0 \;\; \text{for all $j \ge 1$}.
\end{equation}
The Laplace transform of $f$ can be expressed in terms of the Laplace transform of $g_{\sigma}$:
\begin{equation}\label{F_in_terms_of_G}
F(z)=\int_0^{\infty} e^{-z x} f(x)\, \d x=\frac{1}{z}\big(1-G(z)\big),
\end{equation}
where 
\[
G(z):=\int_0^{\infty} e^{-zx} g_{\sigma}(x)\, \d x
\]
is the Laplace transform of $g_{\sigma}$. Formula \eqref{F_in_terms_of_G} follows directly from integration by parts.  

\begin{figure}[t]
\centering
\includegraphics[height =5.5cm]{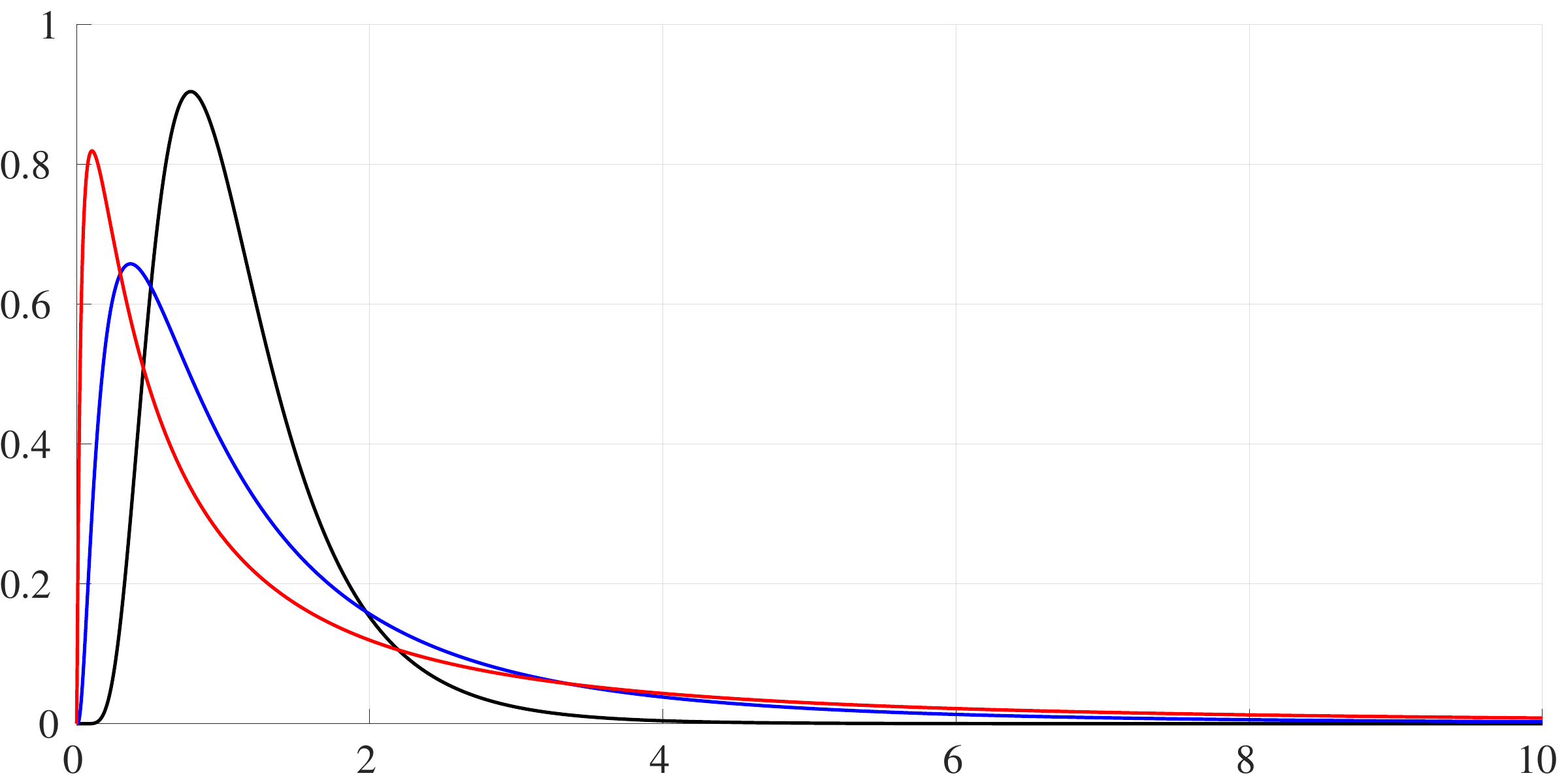} 
\caption{The lognormal pdf $g_{\sigma}(x)$ with $\sigma=0.5$ (black), $\sigma=1$ (blue), and $\sigma=1.5$ (red).} 
\label{fig5}
\end{figure}

Now we can summarize the algorithm for approximating the lognormal pdf $g_{\sigma}(x)$ by exponential sums.  
\begin{enumerate}
\item Compute the values of $G(z_j)$ numerically using the double exponential quadrature \eqref{double_exp_quadrature}.  
\item Use \eqref{F_in_terms_of_G} to obtain the corresponding values of $F(z_j)$.  
\item Define the coefficients $\xi_j$ as in \eqref{xi_j_lognormal}.  
\item Run our continued fraction algorithm to compute the coefficients $c_j$ and $\lambda_j$ of $\phi(x)$.  
\item Finally, form $\widetilde \phi(x)$ according to \eqref{tilde_phi}, which gives the exponential sum approximation to $g_{\sigma}(x)$.  
\end{enumerate}

We constructed exponential sum approximations to $g_{\sigma}$ for $\sigma \in \{0.5,1,1.5\}$. The results are shown in Figure~\ref{fig6}. Note that the horizontal axis corresponds to $\ln(x)$. This logarithmic scale makes it easier to visualize the oscillations in the error $g_{\sigma}(x)-\widetilde \phi(x)$ that occur for very small values of~$x$.  

\begin{figure}[t!]
\centering
\subfloat[][$\sigma=0.5$, $A=1.4$, $B=21$]{\label{fig_6a}\includegraphics[height =5.5cm]{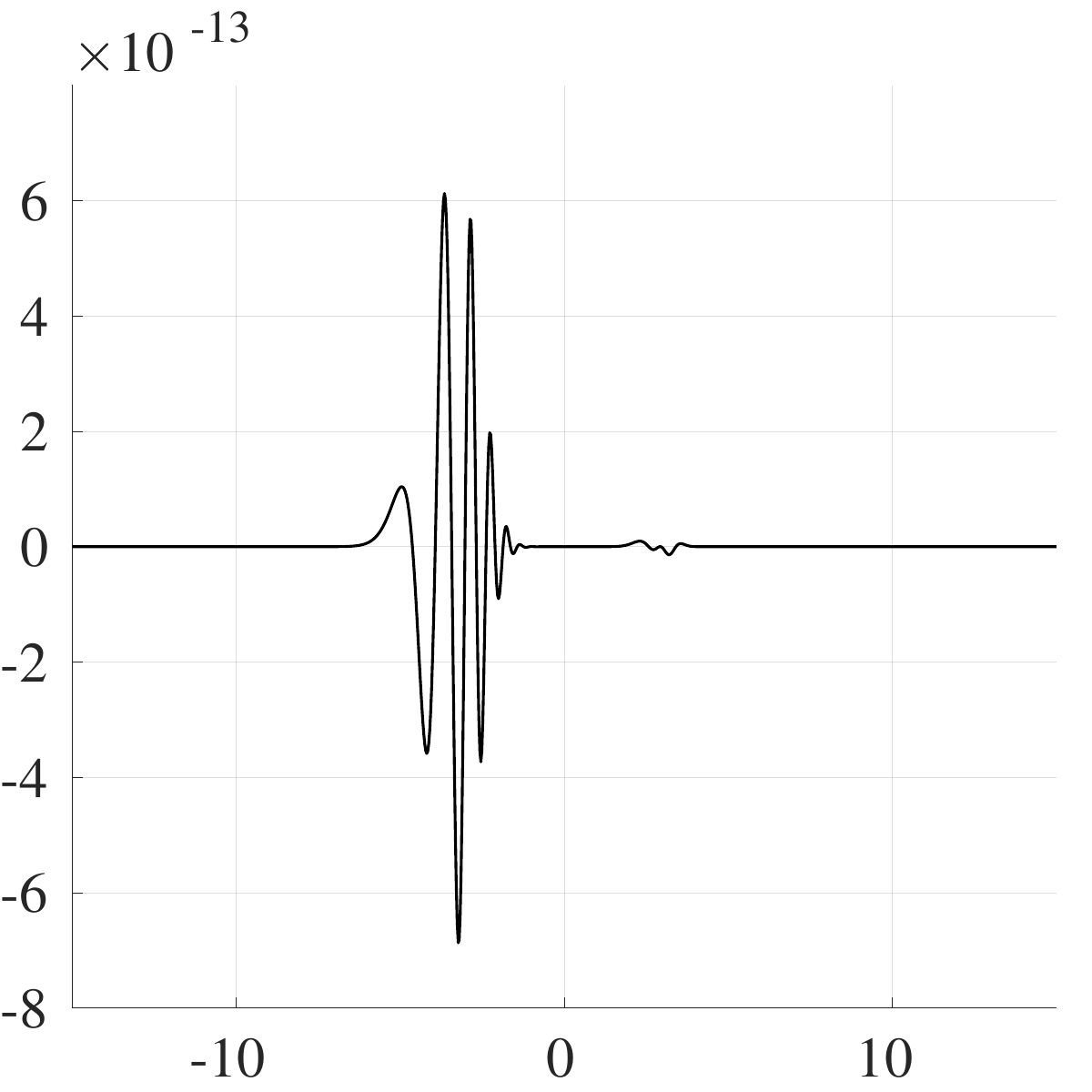}} 
\hspace{0.15cm}
\subfloat[][$\sigma=1$, $A=1.7$, $B=12$]{\label{fig_6b}\includegraphics[height =5.5cm]{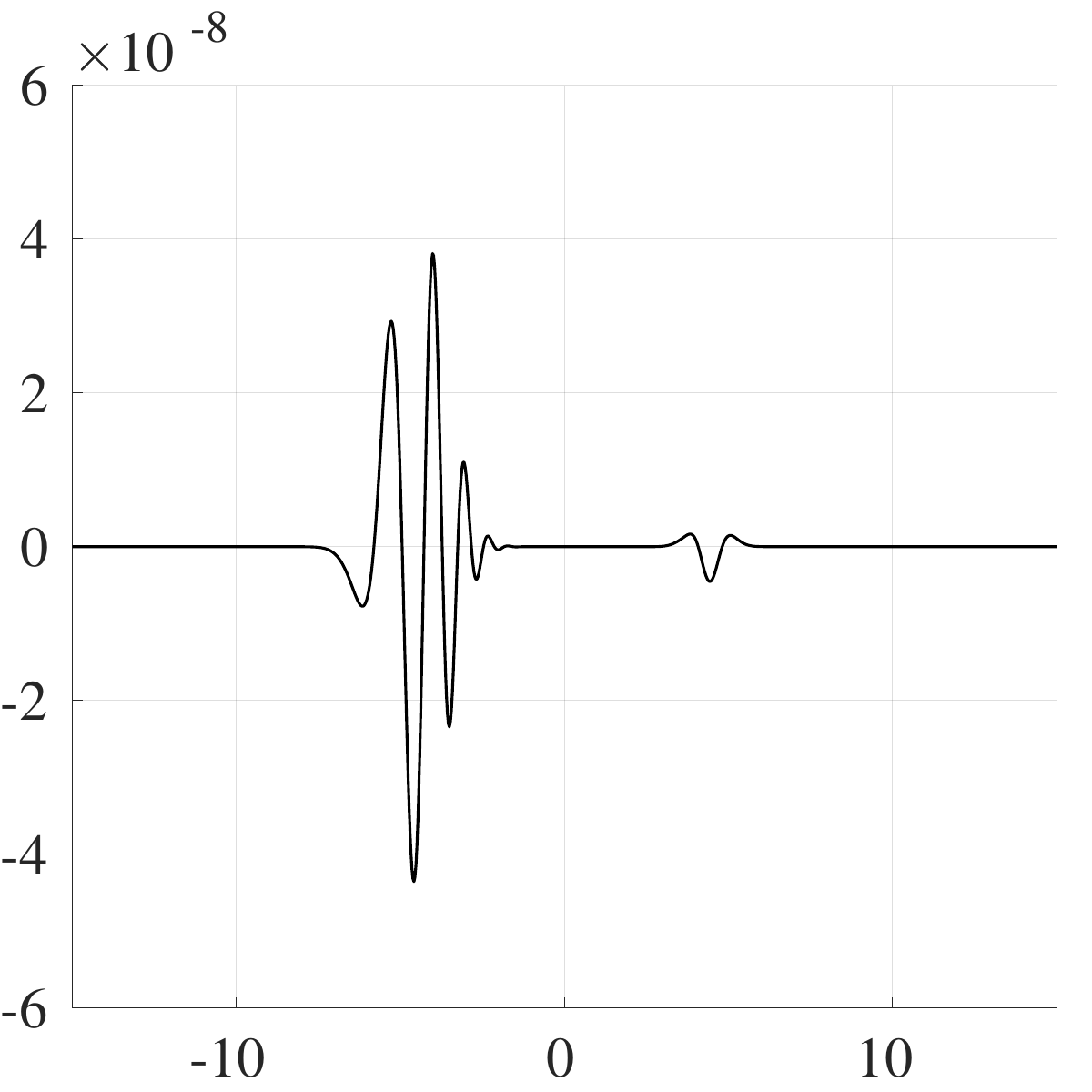}}    
\hspace{0.15cm}
\subfloat[][$\sigma=1.5$, $A=1.1$, $B=8.5$]{\label{fig_6c}\includegraphics[height =5.5cm]{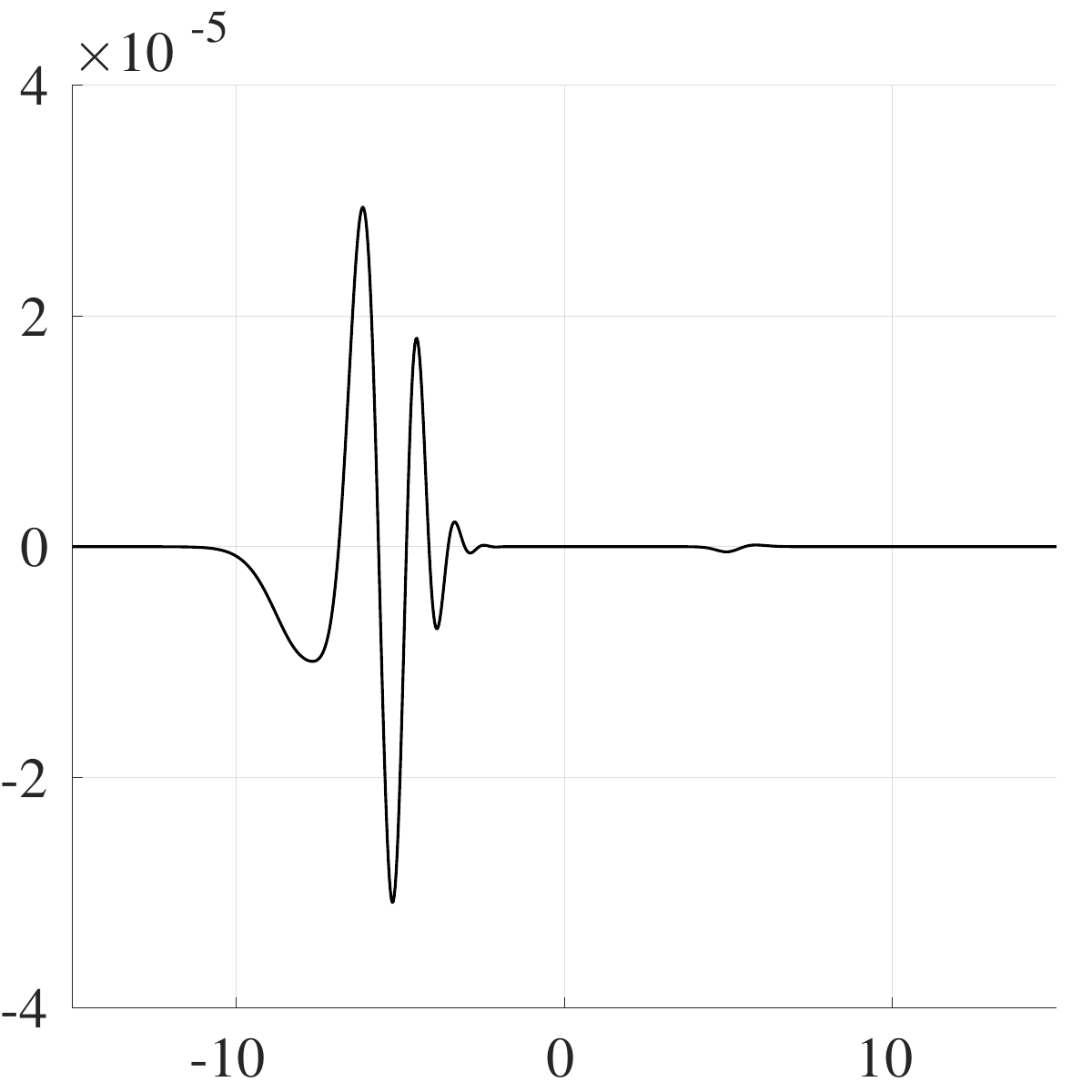}} 
\caption{The graphs of $g_{\sigma}(x)-\widetilde \phi(x)$ with respect to $\ln(x)$. We set $M=30$ in all cases, with $n_{\infty}=6$ in (a), (b), and $n_{\infty}=4$ in (c).} 
\label{fig6}
\end{figure}

\subsection{Approximating the hockey stick and the unit step functions}
\label{subsection_hockeystick}

\begin{figure}[t!]
\centering
\subfloat[][$h(x)$ (black) and $\phi(x)$ (red)]{\label{fig_7a}\includegraphics[height =6.5cm]{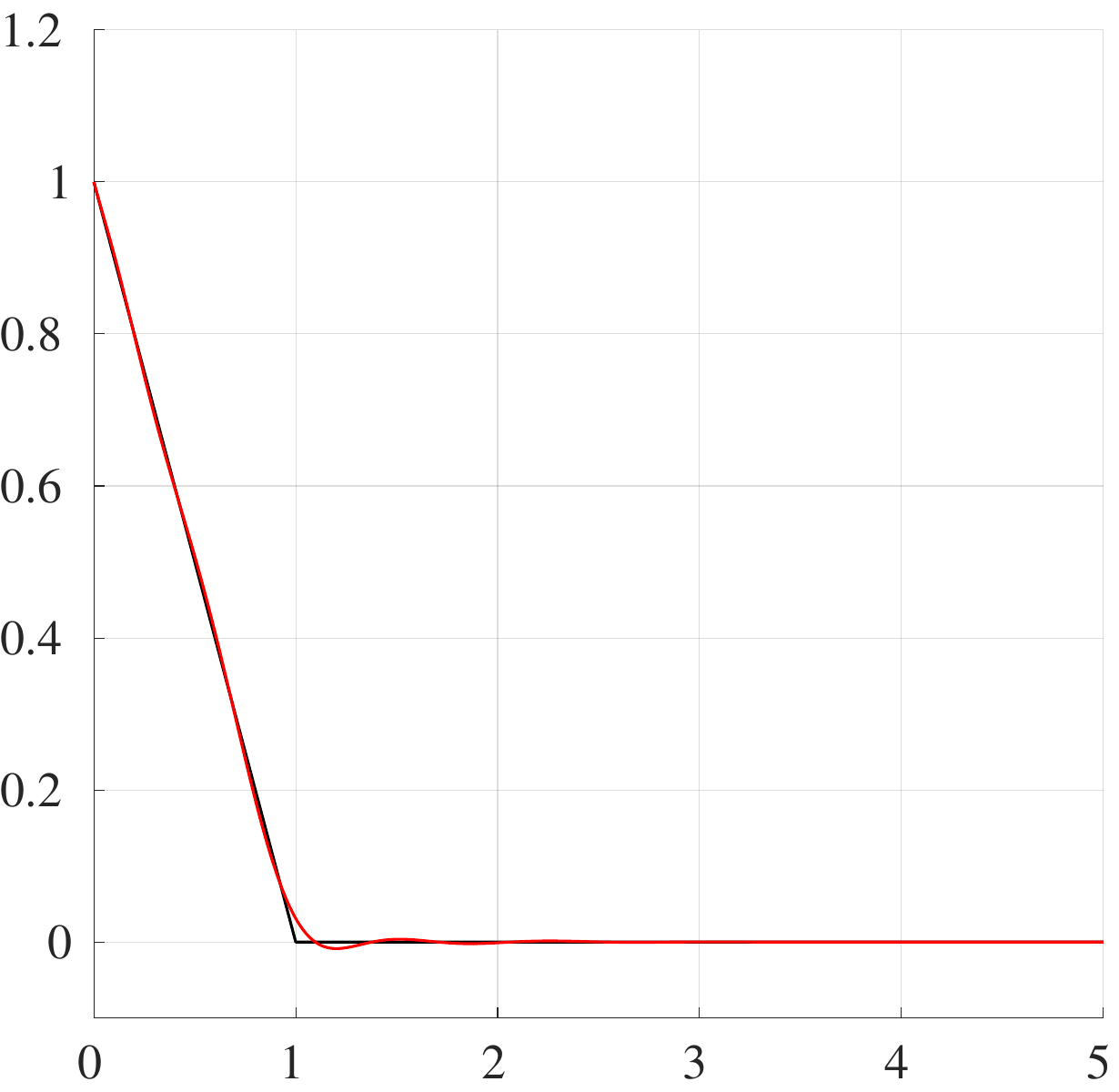}} 
\hspace{0.75cm}
\subfloat[][$h(x)-\phi(x)$]{\label{fig_7b}\includegraphics[height =6.5cm]{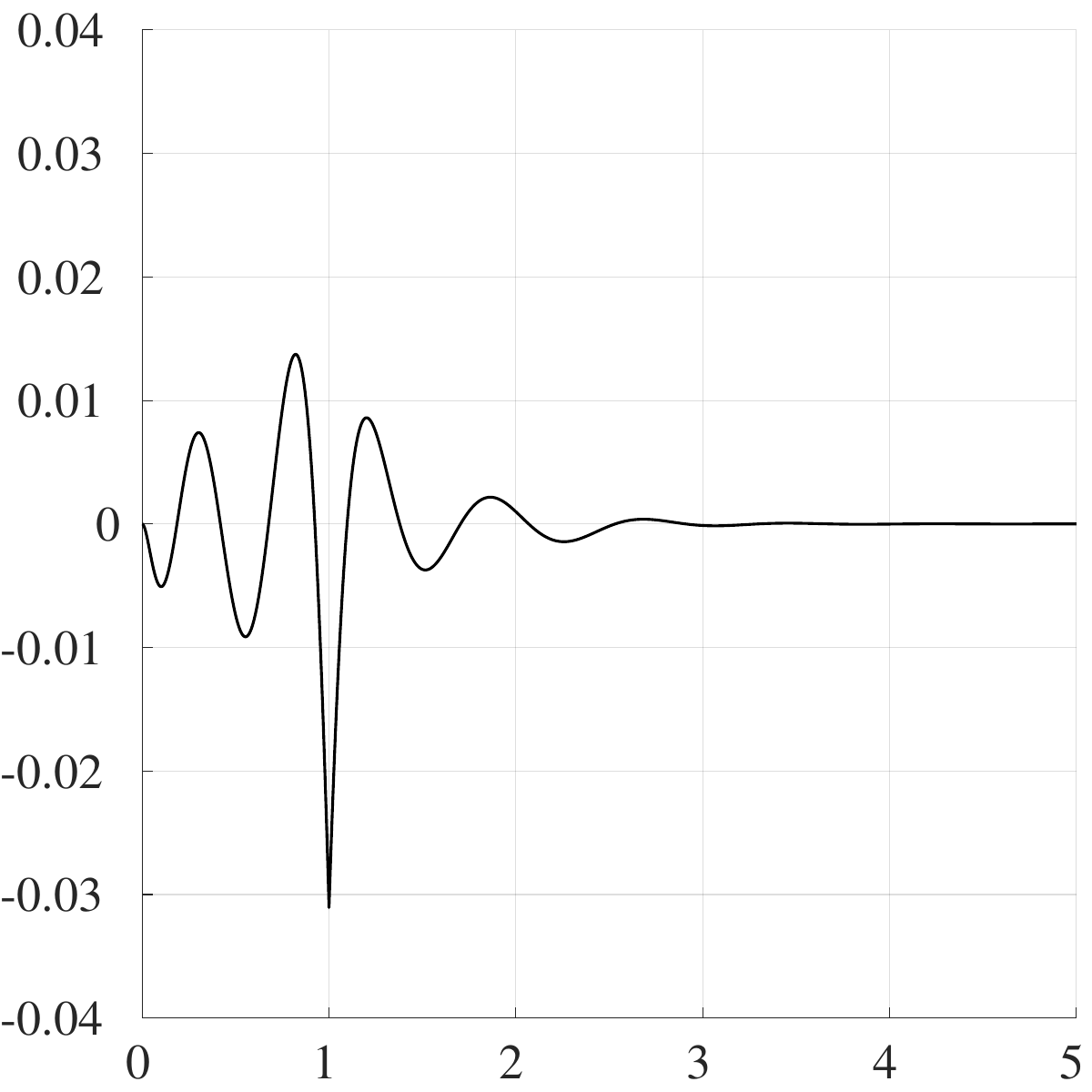}}    
\caption{The graphs of $h(x)$, $\phi(x)$, and $h(x)-\phi(x)$ for $M=5$, $n_{\infty}=2$, $A=0.5$, and $B=8.5$.} 
\label{fig7}
\end{figure}

The hockey stick function is defined as 
\[
h(x):=\max(1-x,0), \qquad x \ge 0.
\]
exponential sum approximations to this function were obtained in \cite{Iscoe_2010} using the Beylkin-Monz\'on algorithm, with applications to pricing financial derivatives such as collateralized debt obligations. This function is particularly well suited to our algorithm. Indeed, the coefficients $\xi_j$ are easily computed since $h(x)=1-x$ for $0\le x \le 1$, giving $\xi_0=1$, $\xi_1=-1$, and $\xi_j=0$ for all $j\ge 2$. The Laplace transform of $h$ is also available in closed form:
\[
F(z)=\int_0^{\infty} e^{-zx } h(x)\, \d x = \int_0^1 e^{-zx} (1-x)\, \d x= z^{-2}\big(e^{-z} + z - 1\big). 
\]

\begin{figure}[t]
\centering
\subfloat[][$M=15$, $n_{\infty}=2$, $B=39$]{\label{fig_8a}\includegraphics[height =5.5cm]{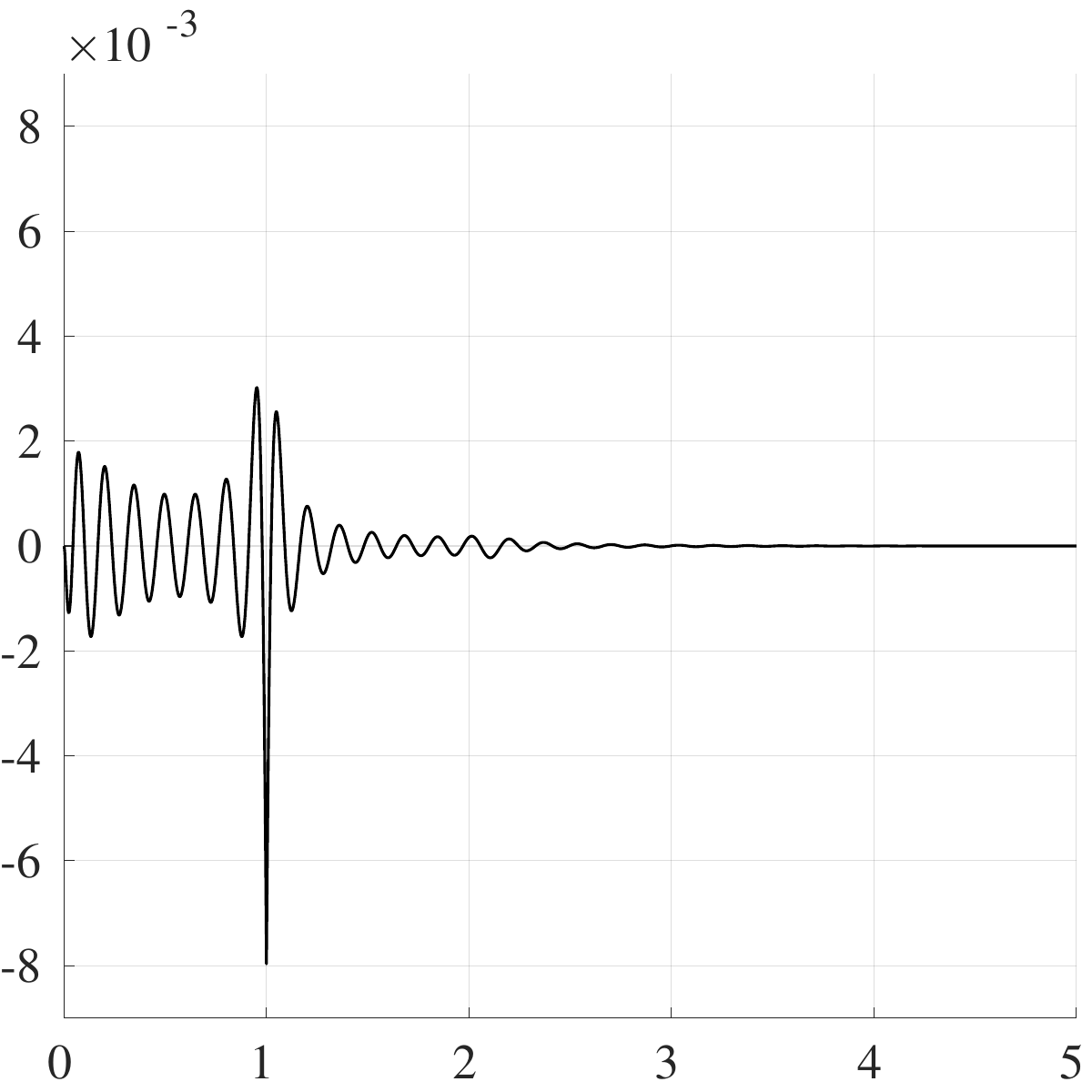}} 
\hspace{0.15cm}
\subfloat[][$M=30$, $n_{\infty}=4$, $B=78$]{\label{fig_8b}\includegraphics[height =5.5cm]{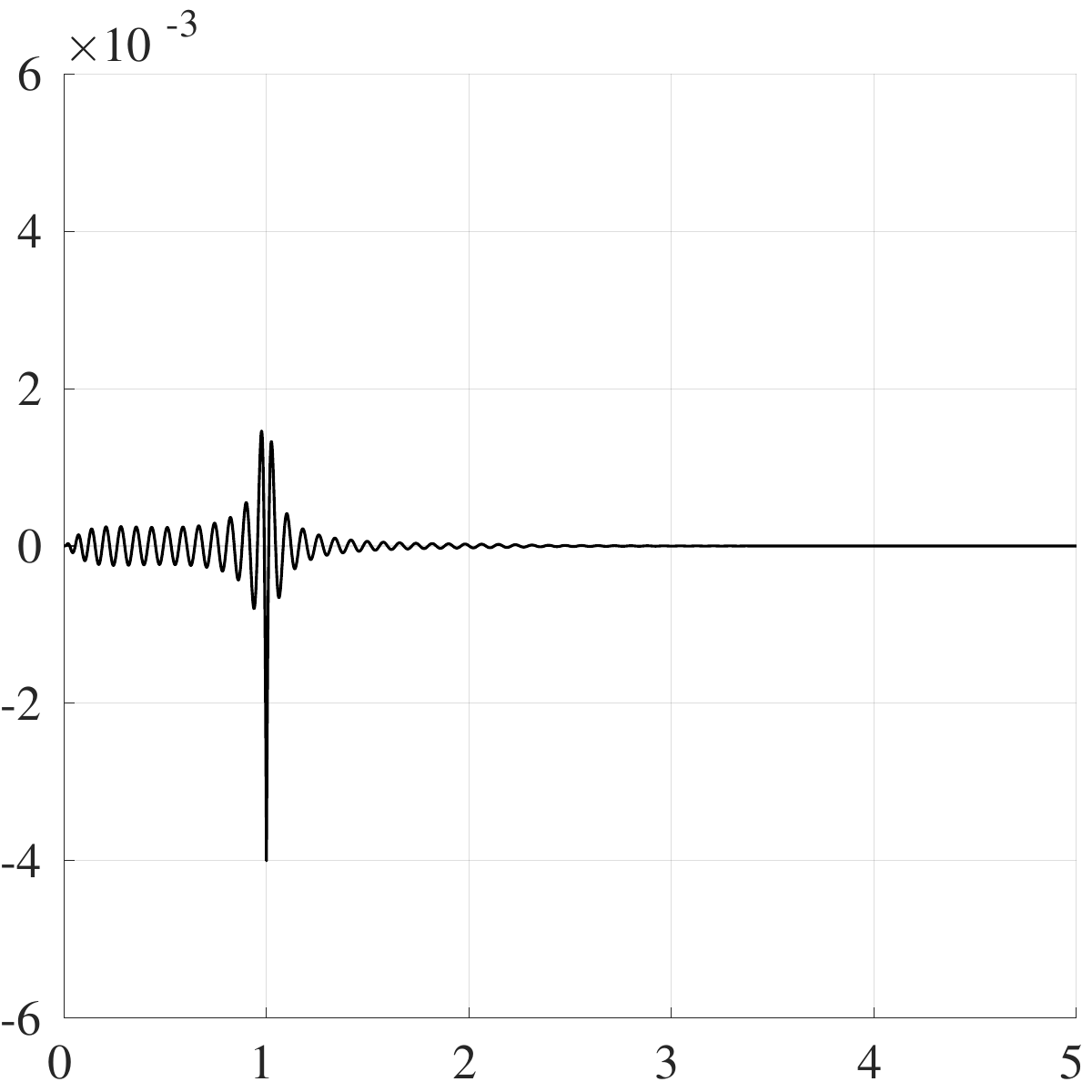}}    
\hspace{0.15cm}
\subfloat[][$M=60$, $n_{\infty}=6$, $B=160$]{\label{fig_8c}\includegraphics[height =5.5cm]{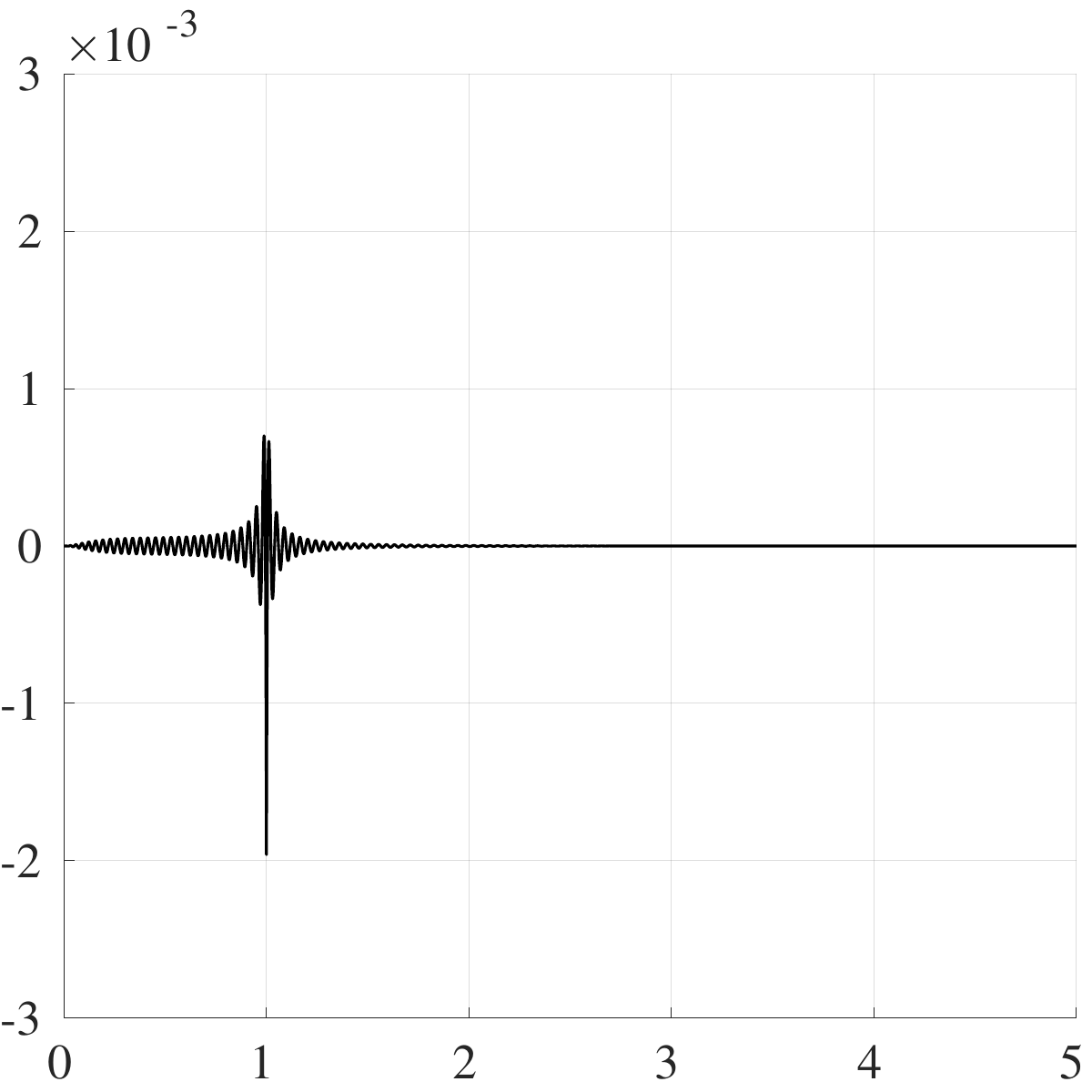}} 
\caption{The graphs of $h(x)-\phi(x)$ for different values of $M$, $n_{\infty}$, and $B$ ($A=0$ in all cases).} 
\label{fig8}
\end{figure}

Figure~\ref{fig7} shows the hockey stick function $h(x)$ and a 5-term exponential sum approximation $\phi(x)$ produced by our algorithm. We also computed approximations with $M\in \{15,30,60\}$ terms. In each case we optimized the parameters $\{A,B,n_{\infty}\}$ to minimize the $L_1$ error $\|h-\phi\|_1$. The errors and parameter choices are displayed in Figures~\ref{fig_8a}--\ref{fig_8c}. The $L_1$ errors for $M\in\{15,30,60\}$ were approximately $1.4\times 10^{-3}$, $3.4\times 10^{-4}$, and $9.7\times 10^{-5}$, respectively.  

A difficulty we encountered is that our algorithm sometimes produced exponential sums with very large coefficients $c_j$. For instance, with the parameters in Figure~\ref{fig_8b}, the largest coefficient $|c_j|$ is about $900$, while the approximation shown in Figure~\ref{fig_8c} had the largest value of $|c_j|$ close to $1.5\times 10^{6}$. Such large values are undesirable since they may cause loss of precision in applications of these approximations. By experimenting with the parameters, we found that increasing $B$ significantly reduces $|c_j|$ without substantially increasing the $L_1$ error. The updated results are shown in Figure~\ref{fig9}. With $M=30$, $n_{\infty}=4$, increasing $B$ from $78$ to $83$ reduced the maximum $|c_j|$ from $900$ to about $56.4$, while the $L_1$ error rose slightly from $3.4\times 10^{-4}$ to $3.8\times 10^{-4}$. For $M=60$, $n_{\infty}=6$, increasing $B$ from $160$ to $175$ reduced the maximum $|c_j|$ from $1.5\times 10^{6}$ to about $97$, while the $L_1$ error rose from $9.7\times 10^{-5}$ to $1.2\times 10^{-4}$. These results appear superior to those obtained in \cite{Iscoe_2010} using the Beylkin-Monz\'on method: in our case, the error is more localized and the $L_{\infty}$ errors are smaller.  

\begin{figure}[t]
\centering
\subfloat[][$M=30$, $n_{\infty}=4$, $B=83$]{\label{fig_9a}\includegraphics[height =6.5cm]{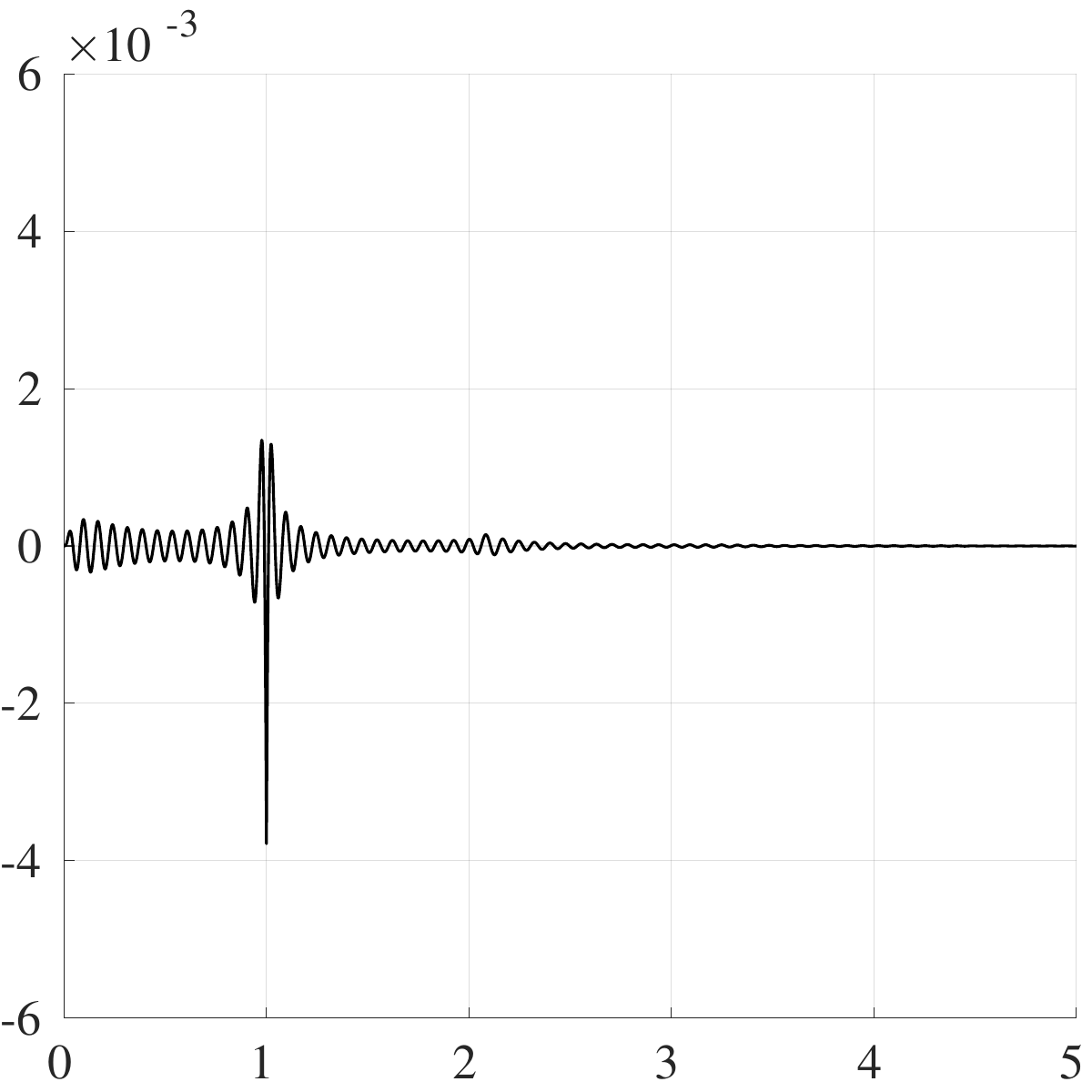}} 
\hspace{0.75cm}
\subfloat[][$M=60$, $n_{\infty}=6$, $B=175$]{\label{fig_9b}\includegraphics[height =6.5cm]{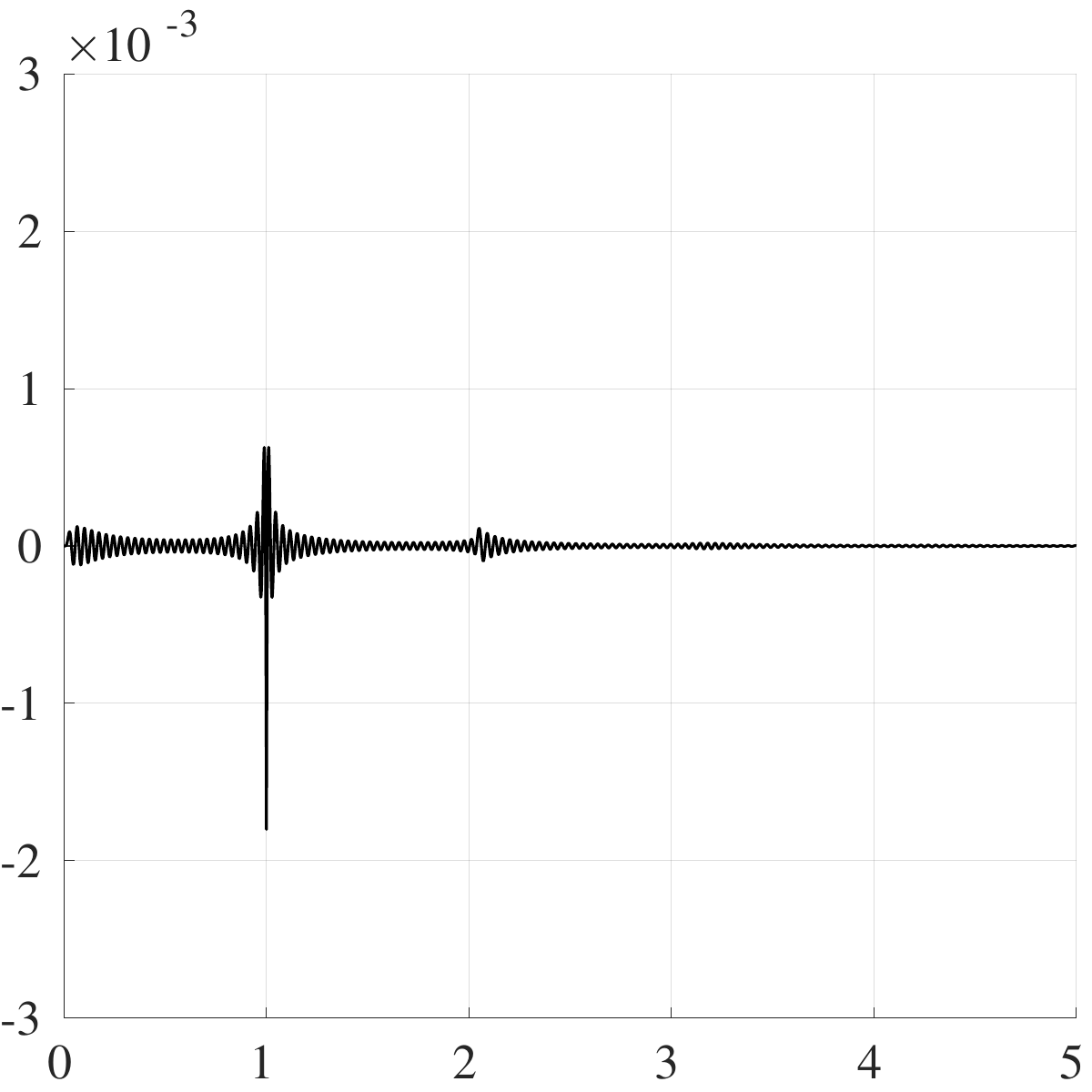}}    
\caption{The graphs of $h(x)-\phi(x)$ for different values of $M$, $n_{\infty}$, and $B$ ($A=0$ in all cases).} 
\label{fig9}
\end{figure}

Next, we consider the unit step function
\[
H(x):={\mathbf 1}_{\{x \le 1\}}, \qquad x\ge 0.
\]
Note that $H(x)=-h'(x)$, so exponential sum approximations to $h$ yield approximations to $H$ (as in Section~\ref{subsection_Lognormal}). However, since $H(0+)\neq 0$, our algorithm can also be applied directly to $H$, which we do in the following experiments.  The Laplace transform of $H(x)$ is easily seen to be equal to $(1-\exp(-z))/z$. 

Figure~\ref{fig_10a} shows $H(x)$ together with a 15-term exponential sum approximation $\phi(x)$, and Figure~\ref{fig_10b} shows the error $H(x)-\phi(x)$. Figures~\ref{fig_10c} and \ref{fig_10d} show the errors for approximations with $30$ and $60$ terms. In each case we set $A=0$ and chose $n_{\infty}$ and $B$ to minimize the $L_1$ error while keeping all coefficients $|c_j|<100$.  

\begin{figure}[t]
\centering
\subfloat[][$M=15$, $n_{\infty}=4$, $B=34$ ]{\label{fig_10a}\includegraphics[height =6.5cm]{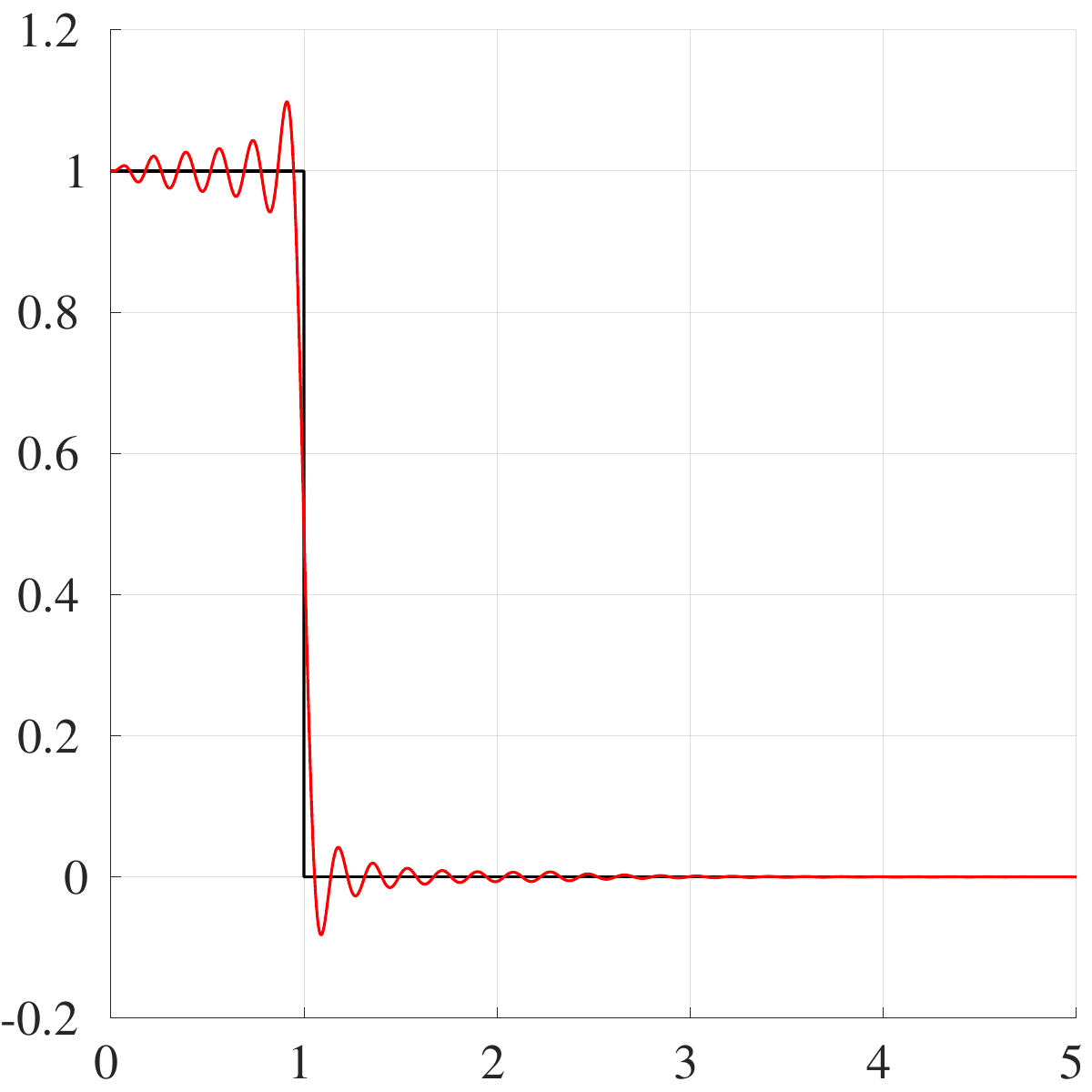}}
\hspace{0.75cm}
\subfloat[][$M=15$, $n_{\infty}=4$, $B=34$]{\label{fig_10b}\includegraphics[height =6.5cm]{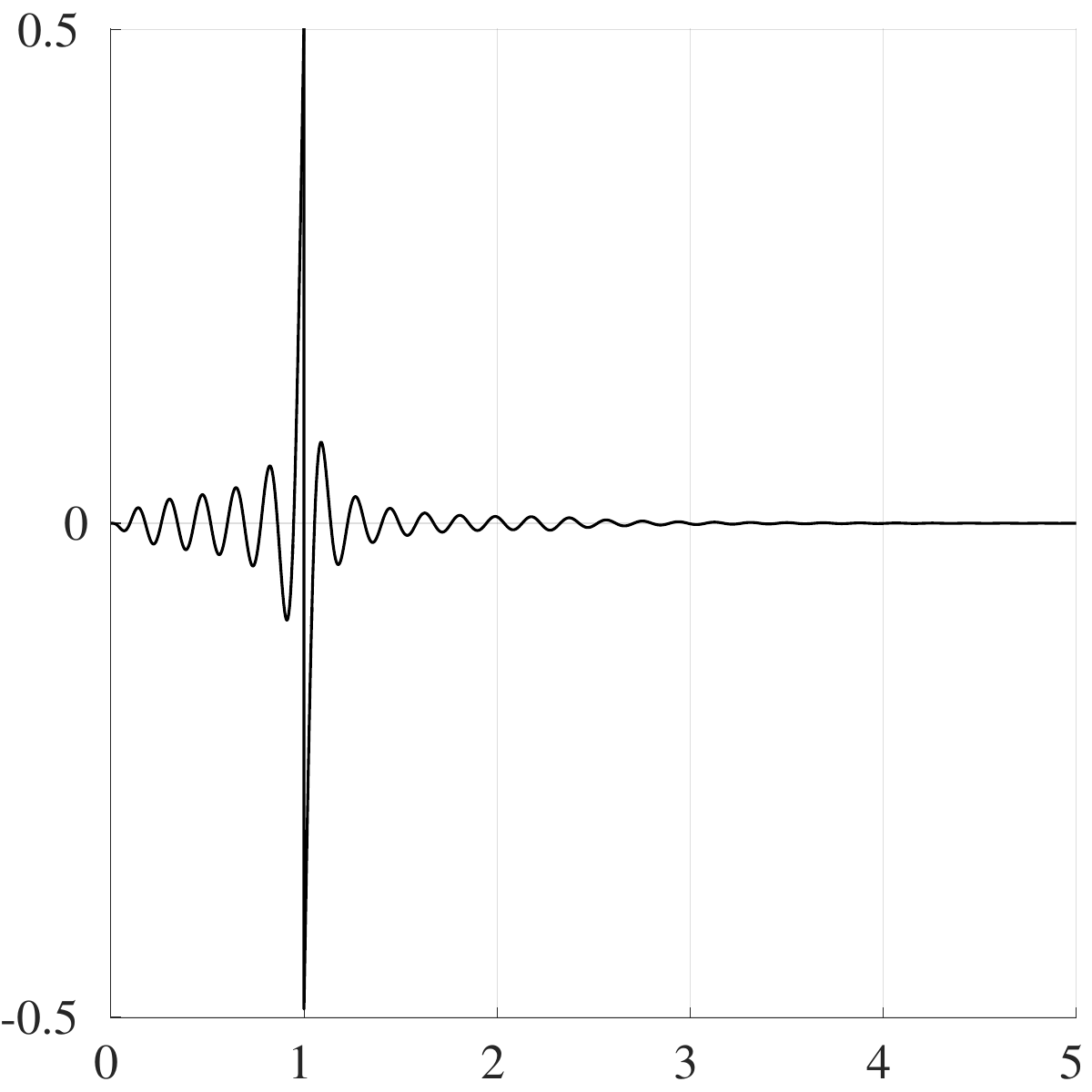}} 
\\
\subfloat[][$M=30$, $n_{\infty}=8$, $B=75.5$]{\label{fig_10c}\includegraphics[height =6.5cm]{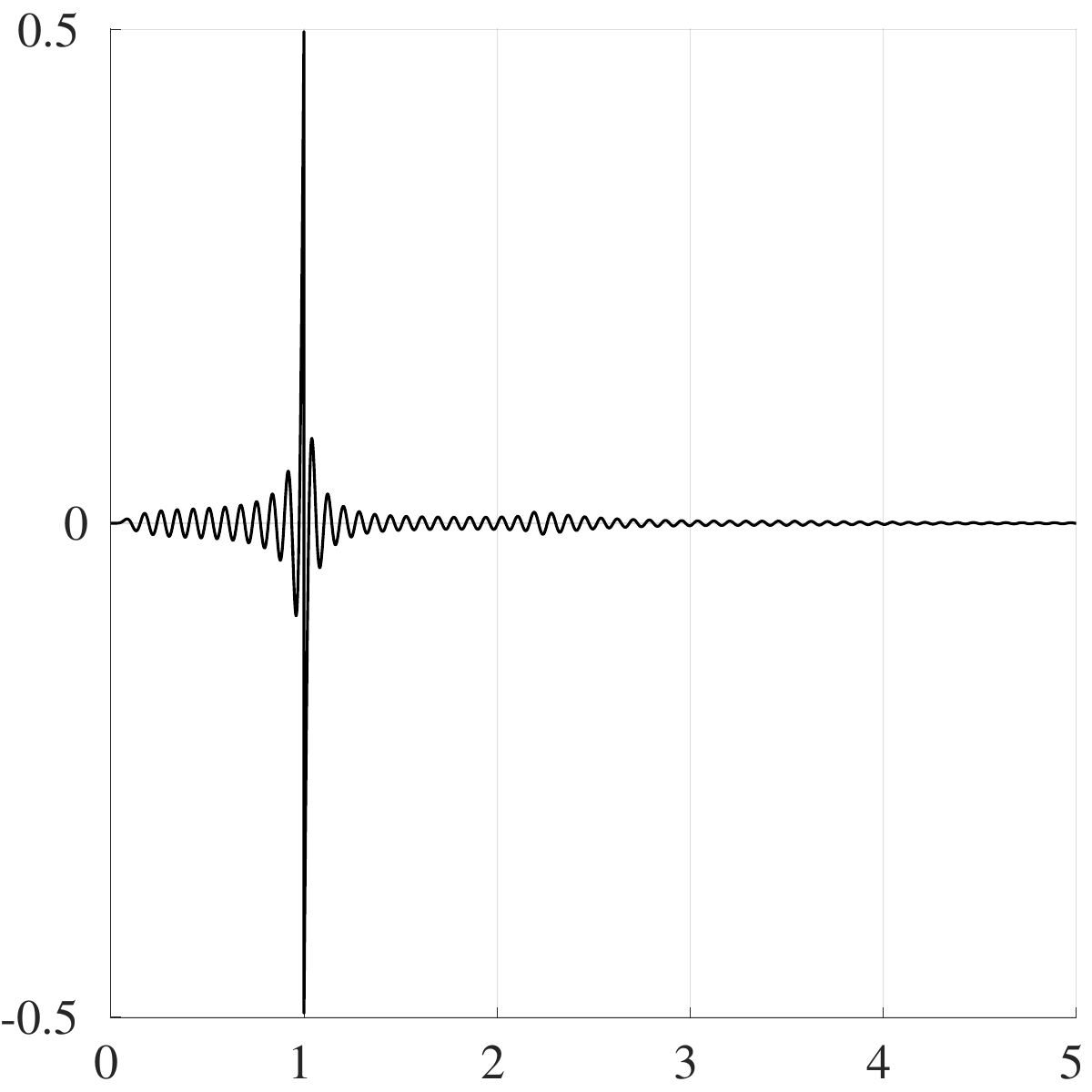}} 
\hspace{0.75cm}
\subfloat[][$M=60$, $n_{\infty}=16$, $B=158$]{\label{fig_10d}\includegraphics[height =6.5cm]{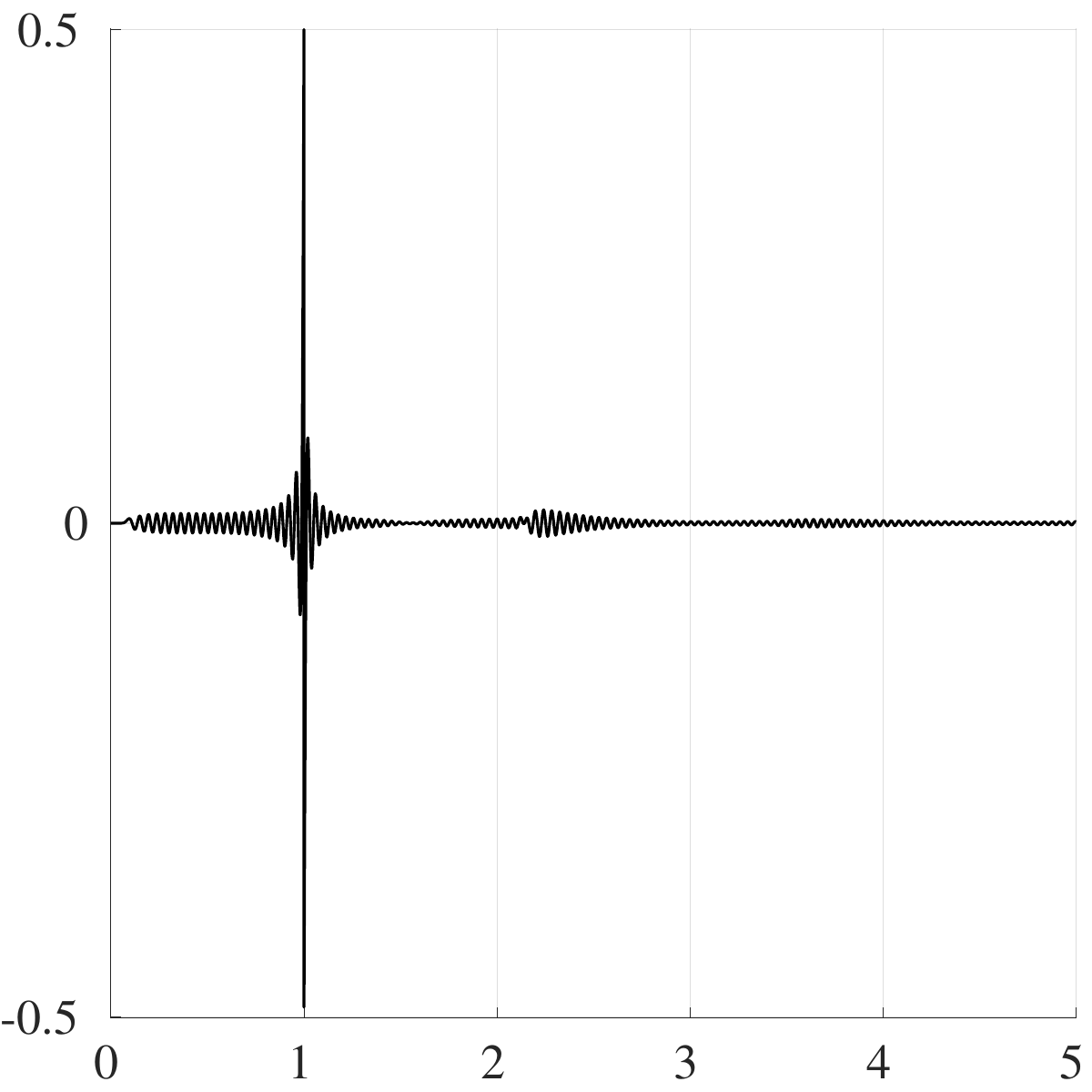}}  
\caption{(a) The unit step function $H(x)$ (black) and its exponential sum approximation $\phi(x)$ (red). (b)--(d) The errors $H(x)-\phi(x)$ for different values of $M$, $B$, and $n_{\infty}$. In all cases we set $A=0$.} 
\label{fig10}
\end{figure}

One possible application of exponential sum approximations to $H(x)$ arises in Laplace transform inversion. A relevant example from probability is the computation of the cumulative distribution function (CDF) of a positive random variable $X$ when only its Laplace transform $F_X(z)=\e[\exp(-zX)]$ is known. This situation occurs frequently, for instance, when computing the CDF of a sum of independent random variables. The CDF of $X$ can be expressed as
\[
\p(X \le u)=\e [ {\mathbf 1}_{\{X/u \le 1\}}]=\e [ H(X/u)] \approx 
\e \!\left[ \sum_{j=1}^M c_j e^{-\lambda_j X /u} \right]
=\sum_{j=1}^M c_j F_X(\lambda_j/u), \qquad u\ge 0. 
\]
We leave the investigation of the efficiency of this method, as well as its comparison with other Laplace transform inversion techniques, to future work.

\section{Conclusion and future work}\label{Section_conclusion}

In this paper we introduced an algorithm for finding approximations to real-valued functions on $\r^+$ by exponential sums. The algorithm relies on the continued fraction approach to find a rational approximation to the Laplace transform of $f$, and the coefficients of the exponential sum are then recovered from this rational approximation via partial fraction decomposition. While we are not able to offer any rigorous proofs concerning the convergence or accuracy of this method, the numerous numerical experiments presented in this paper suggest that this algorithm performs well and deserves further study. 

We conclude by outlining several possible extensions and improvements of the algorithm. One immediate extension is that we could relax the condition that  $f \in L_1((0,\infty), \d x)$ and, instead, require only that the Laplace transform $F$ (defined in \eqref{def:F_Laplace_transform}) exists in some half-plane $\re(z)\ge c$. By placing all interpolation points $z_j$ in this half-plane, our algorithm would work without any modification. Equivalently, we could apply our algorithm (as described in this paper) to the function $\tilde f(x)=f(x) e^{-c x}$ and then multiply the resulting exponential sum $\phi$ by $e^{c x}$ to obtain an exponential sum approximation to $f$. 

As an illustration, we apply our algorithm to the function $f(x)=1/(1+x)$. The Laplace transform exists for all $\re(z)>0$ and when $p$ is even all $z_j$ lie in the half-plane $\re(z)>0$. We fixed $n_{\infty}=2$ and for each $M\in \{4,6,8\}$ we searched for $A$ and $B$ that would produce accurate exponential sum approximations to $f(x)=1/(1+x)$ in the sense of the $L_{\infty}([0,1],\d x)$ norm. The results of this experiment are shown in Figure \ref{fig12}. While the results are not as good as those obtained in \cite{Hackbusch_2019} (where the best $L_{\infty}$ approximations were found), 
they still produce good approximations and suggest exponential decay of error with increasing number of terms in the exponential sum $\phi$. Interestingly, the original function $f(x)=1/(1+x)$ is completely monotone and the approximations found by our algorithm seem to stay in the same class of completely monotone functions: the coefficients $c_j$ and $\lambda_j$ produced by our algorithm are all positive . 

\begin{figure}[t]
\centering
\subfloat[][$M=4$, $A=8$, $B=5.5$]{\label{fig_12a}\includegraphics[height =5.5cm]{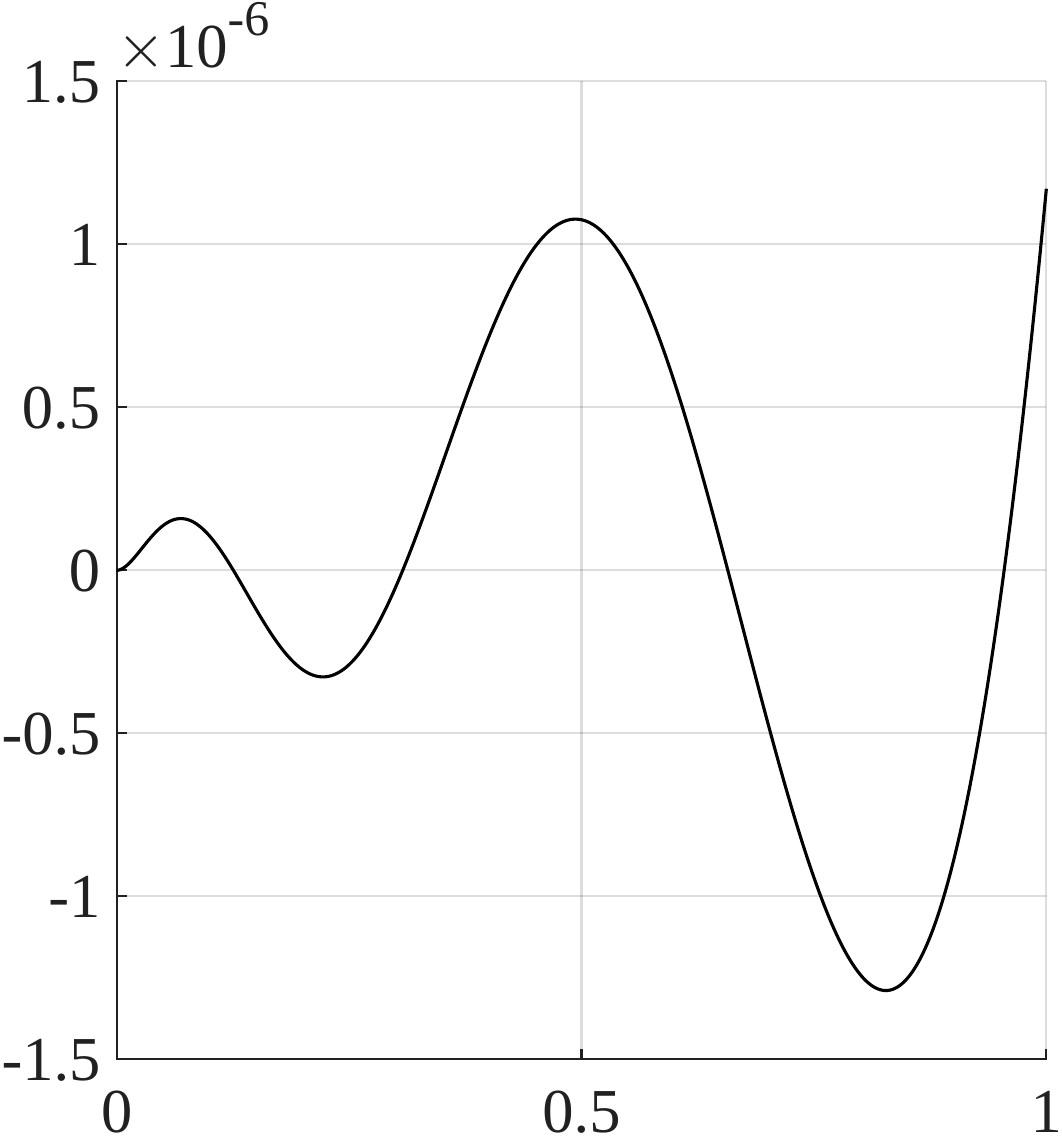}} 
\hspace{0.15cm}
\subfloat[][$M=6$, $A=11.5$, $B=8.5$]{\label{fig_12b}\includegraphics[height =5.5cm]{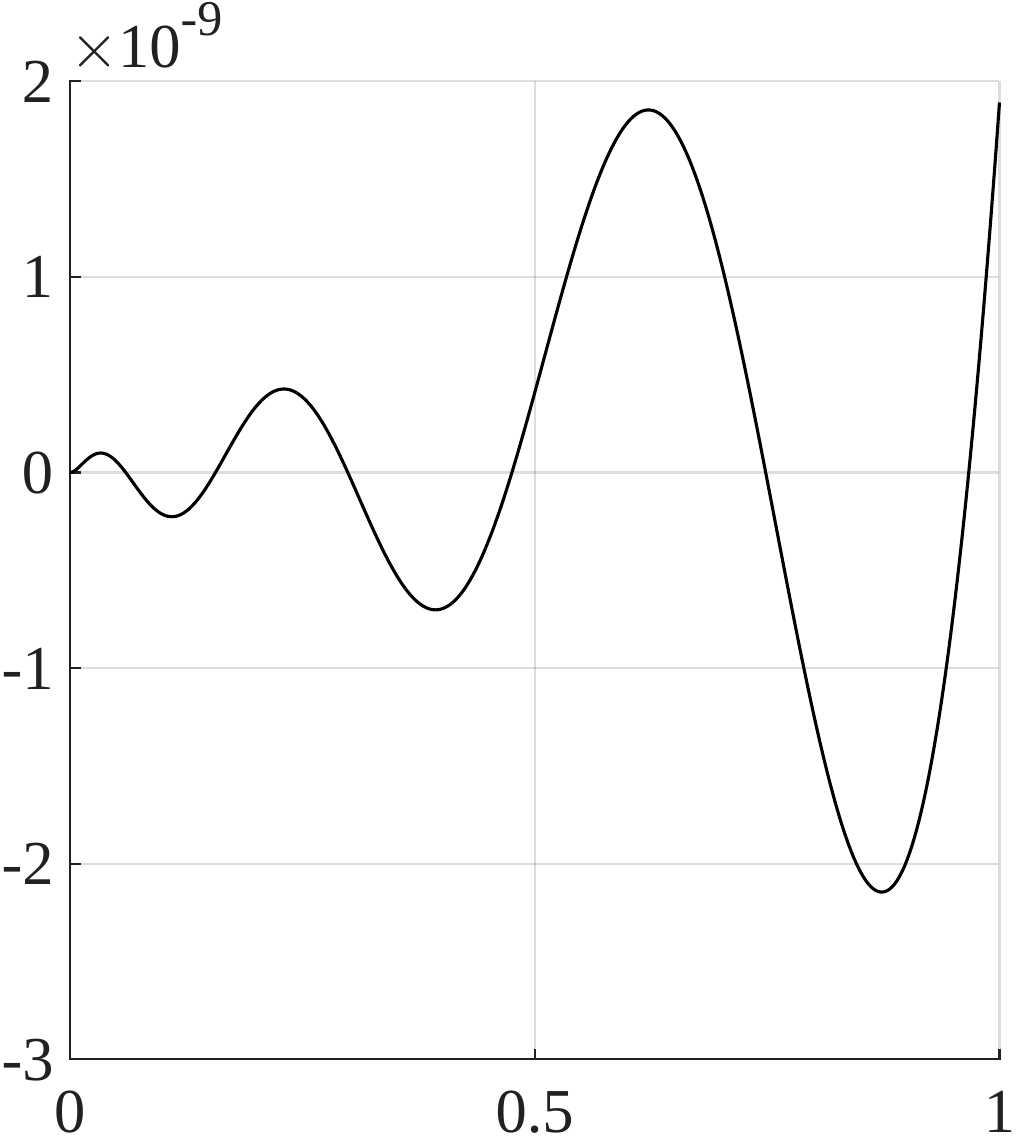}}    
\hspace{0.15cm}
\subfloat[][$M=8$, $A=14.5$, $B=11.5$]{\label{fig_12c}\includegraphics[height =5.5cm]{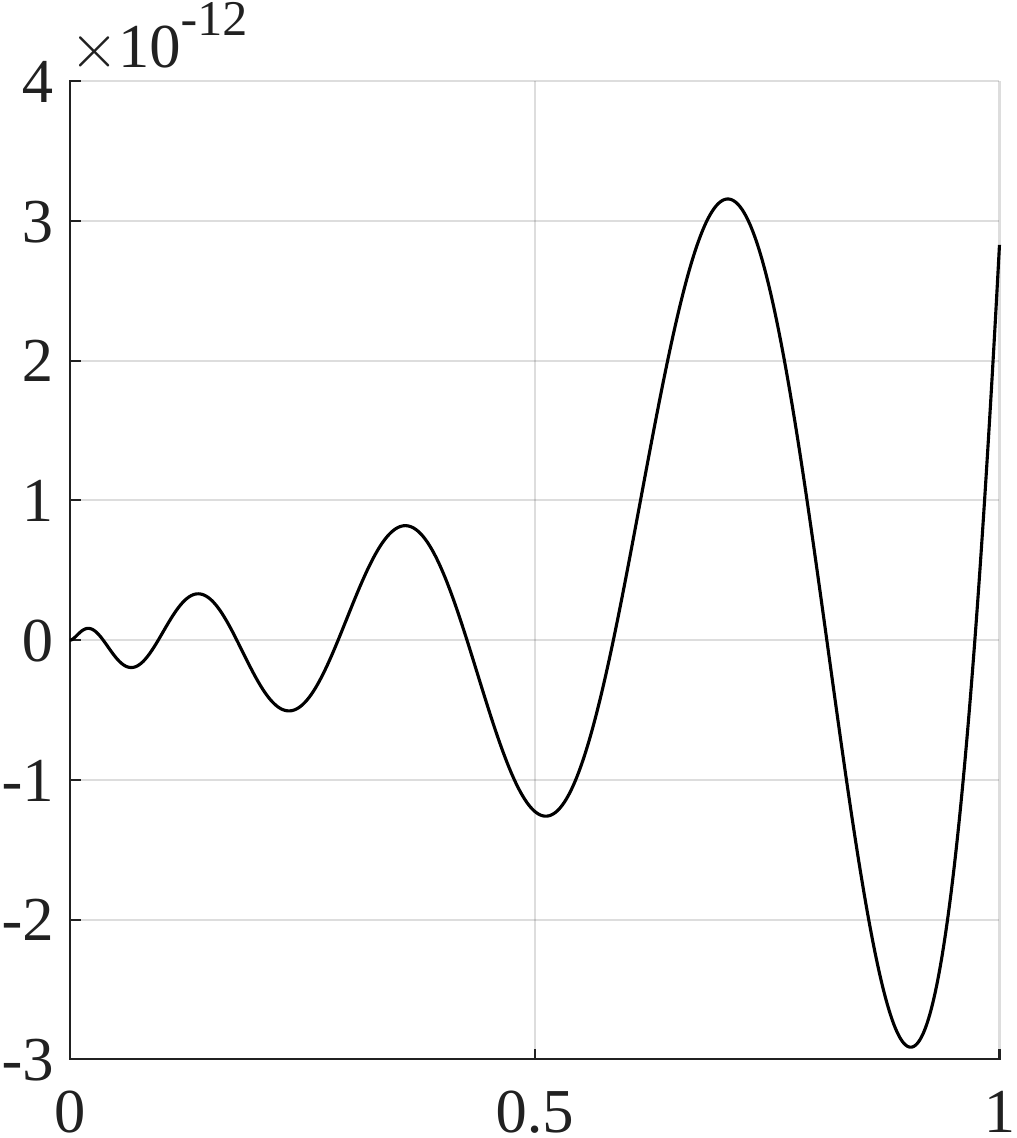}} 
\caption{The graphs of $(1+x)^{-1}-\phi(x)$ for different values of $M$, $A$ and $B$ ($n_{\infty}=2$ in all cases).} 
\label{fig12}
\end{figure}

Another direction that needs to be explored is the effect of the choice of the interpolation points $z_j$ on the resulting exponential sum approximation $\phi$. In this paper we chose to position $z_j$ (in an equidistant and symmetrical fashion) on two rays from the origin. The reasons for this choice are that it is simple (requires only two parameters $A$ and $B$) and seems to produce good exponential sum approximations $\phi$. However, one could place $z_j$ on different curves: other natural choices would be to position $z_j$ on parabolas $x=C y^2$ or on more general curves $x= C y^{\alpha}$. It could also be the case that positioning $z_j$ in an equidistant manner is not optimal. Perhaps  positioning the points $z_j$ more densely near the origin would produce better results. It remains to be seen what effect these modifications would have on the accuracy of the resulting exponential sum approximations. 

%****************************************************************************************************************
%****************************************************************************************************************
%****************************************************************************************************************

\section*{Data availability}

The code and computation results are available at \href{https://github.com/Alexey-Kuznetsov-math/exponential_sum_approximations}{github} or at authors website \url{https://kuznetsovmath.ca/}.

\section*{Acknowledgements}
The research was supported by the Natural Sciences and Engineering Research Council of Canada. We
are grateful to anonymous referees for carefully reading the paper and for providing very helpful comments.

%\bibliographystyle{abbrv}
%\bibliography{references}

\end{document}